\documentclass[reqno, 12pt]{amsart}

\usepackage{mathrsfs}
\usepackage{amscd}
\usepackage{amsmath}
\usepackage{latexsym}
\usepackage{amsfonts}
\usepackage{amssymb}
\usepackage{amsthm}
\usepackage{graphicx}
\usepackage{amsthm, amsfonts, amssymb, color}
\usepackage{array}
\usepackage{amstext, amsxtra}
\usepackage{txfonts}
\usepackage{hyperref}

\newtheorem{theorem}{Theorem}[section]
\newtheorem{proposition}[theorem]{Proposition}
\newtheorem{corollary}[theorem]{Corollary}
\newtheorem{lemma}[theorem]{Lemma}
\newtheorem{definition}[theorem]{Definition}

\newtheorem{remark}[theorem]{Remark}

\numberwithin{equation}{section}

\title[Dispersive estimates]
{Dispersive estimates for inhomogeneous fourth-order Schr\"odinger operator in 3D with zero energy obstructions }
\author{Hongliang Feng}

\address {Hongliang Feng, School of Mathematics Sciences, Chongqing Normal University, Chongqing, 401131, P.R. China}
\email{fenghl@cqnu.edu.cn}
\date{\today}

\begin{document}
\maketitle
\begin{abstract}
We study the $L^1-L^\infty$ dispersive estimate of the inhomogeneous fourth-order Schr\"{o}dinger operator $H=\Delta^{2}-\Delta+V(x)$ with zero energy obstructions in $\mathbf{R}^{3}$.  For the related propagator $e^{-itH}$,  we prove that  for  $0<|t|\leq 1$, then $e^{-itH}P_{ac}(H)$ satisfies the $|t|^{-3/4}$-dispersive estimate.  For $|t|>1$, we prove that:\,\, 1) if zero is a regular point of $H$, then  $e^{-itH}P_{ac}(H)$ satisfies the $|t|^{-3/2}$- dispersive estimate.\,\, 2) if zero is purely a resonance of $H$,  there exists a time dependent operator $\mathcal{F}_{t}$  such that  $e^{-itH}P_{ac}(H)-\mathcal{F}_{t}$ satisfies the $|t|^{-3/2}$- dispersive estimate.\,\, 3) if zero is purely an eigenvalue or zero is both an eigenvalue and a resonance of $H$, then there exists a time dependent  operator $\mathcal{G}_{t}$  such that  $e^{-itH}P_{ac}(H)-\mathcal{G}_{t}$ satisfies the $|t|^{-3/2}$-dispersive estimate. Here $\mathcal{F}_{t}$ and $\mathcal{G}_{t}$ satisfy the $|t|^{-1/2}$-dispersive estimate.
\end{abstract}

\tableofcontents

\section{Introduction}
\setcounter{equation}{0}

In this paper we consider the $L^1-L^{\infty}$ dispersive estimates of the inhomogeneous fourth-order Schr\"odinger operator
$$ H=\Delta^2-\Delta+V(x),\,\, H_{0}=\Delta^{2}-\Delta,\,\, x\in\mathbf{R}^3$$
in $L^{2}(\mathbf{R}^{3})$, where $V(x)$ is a real-valued function satisfies $|V(x)|\lesssim(1+|x|)^{-\beta}$ with some $\beta>0$.

The considered operator originate from the following fourth order Schr\"odinger equation
\begin{equation}\label{IFNLS}
i\partial_{t}u+\Delta u+\varepsilon\Delta^2 u+|u|^{2p}u=0,\,\, u: \mathbf{R}\times\mathbf{R}^{d}\rightarrow \mathbf{C},\,\, \varepsilon\in\mathbf{R}.
\end{equation}
Equation \eqref{IFNLS} was introduced by Karpman \cite{Karpman-1994, Karpman-1996}  and Karpman and Shagalov \cite{Karpman-Shagalov-2000} to take into account the role of small fourth-order dispersion terms in the propagation of intense laser beams in a bulk medium with Kerr nonlinearity. When $\varepsilon=0, d=2$ and $p=1$, this corresponds to the canonical model. When $2<dp<4$,  Karpman and Shagalov \cite{Karpman-Shagalov-2000} showed, among other things, that the waveguides induced by the nonlinear Schr\"odinger equation become stable when $|\varepsilon|$ is taken sufficiently large. Then  equation  \eqref{IFNLS}  is predominantly governed by the corresponding homogeneous fourth-order equation
\begin{equation}\label{FNLS}
i\partial_{t}u+\varepsilon \Delta^2 u+|u|^{2p}u=0,\,\,\, \varepsilon\in\mathbf{R}\setminus\{0\}.
\end{equation}
The homogeneous fourth-order Schr\"odinger equation \eqref{FNLS}  has been  widely investigated in \cite{Miao-Zheng-FNLS-homo-2016, Hayashi-Navarro-Naumkin-JDE-2016, Hayashi-Naumkin-ZAMP-2015, Miao-Xu-Zhao-FNLS-homo-2011, Pausader-Shao-homoFNLS-2010,  Miao-Xu-Zhao-FNLS-homo-2009, Fibich-Ilan-Papanicolaou-2002} and references therein.  The inhomogeneous fourth-order equation  \eqref{IFNLS} brings some interesting questions because it does not enjoy scaling invariance. The problem \eqref{IFNLS} has been recently investigated by \cite{Saut-Segata-JMAA-2020, Saut-Segata-DCDS-2019, Ruzhansky-Wang-Zhang-inhomoFNLS-2016, Natali-Pastor-SIAM-2015, Segata-inhomoFNLS-2015, Hayashi-Naumkin-JMP-2015, Pausader-Xia-inhomoFNLS-2013, Segata-inhomoFNLS-2011, Jiang-Pausader-Shao-inhomoFNLS-2010} and references therein.

In this paper, we are interesting in the inhomogeneous fourth-order operator $H=\Delta^2-\Delta+V(x)$ with decaying potential in 3D. We devote to establish the $L^{1}-L^{\infty}$ dispersive estimate for $e^{-itH}$ which plays key role in related nonliner dispersive problems, see e.g. \cite{TT2, Cazenave}.  For the homogeneous fourth-order operator $\Delta^2+V$,  linear dispersive estimates have recently been  studied in \cite{Burak-Green-Toprak-3D, Green-Toprak-4D, FSY, FSWY-TAMS-2020}. In \cite{FSY, FSWY-TAMS-2020}, they obtained the Kato-Jensen decay estimates of $e^{-it(\Delta^2+V)}$ with presence of zero resonance or zero eigenvalue in $\mathbf{R}^{d}$ with $d\geq5$.  In \cite{Burak-Green-Toprak-3D, Green-Toprak-4D}, they established the $L^1-L^{\infty}$ dispersive estimates of $e^{-it(\Delta^2+V)}$ with presence of zero resonance or zero eigenvalue in three and four dimensions respectively. However, for $d=1, 2$ and $d\geq5$, the $L^1-L^{\infty}$ dispersive estimates of $e^{it(\Delta^2+V)}$ is still open up to my knowledge.  The $L^1-L^{\infty}$  dispersive estimates of the inhomogeneous fourth-order operators $\Delta^2\pm\Delta+V$ with nonvanish potential has  very few results so far.    In this sequel, we will restruct our attention to the $L^1-L^{\infty}$ dispersive estimates of  $H=\Delta^2-\Delta+V$ in 3-dimension with zero energy obstructions. In the coming paper, we will study  the $L^1-L^{\infty}$ dispersive estimates for the inhomogeneous operator $\Delta^2+\Delta+V$. The operator $\Delta^2+\Delta+V$ is more complicated than $H=\Delta^2-\Delta+V$.  It has two finite thresholds $-1/4$ and $0$ while $H=\Delta^2-\Delta+V$ has only zero threshold.

For free propagators $e^{-it(\Delta^2-\varepsilon\Delta)}$ with $\varepsilon\in\{-1, 0, 1 \}$,  Ben-Artzi, Koch and Saut  in \cite{BenArtzi-Koch-Saut-2000} established the pointwise estimates for  kernels of $e^{-it(\Delta^2-\varepsilon\Delta)}$.
When $\epsilon=1$, let $I_{1}(t, x)$ be the kernel of $e^{-it(\Delta^2-\Delta)}$.  If $0<t\leq1$ or $|x|\geq t$, then
\begin{equation*}
  \big|D^{\alpha}I_{1}(t, x)\big|\leq ct^{-(d+|\alpha|)/4}\ \Big(1+t^{-1/4}|x|\Big)^{(|\alpha|-d)/3},\,\, \alpha\in\mathbf{N}^{d},\,\,  x\in\mathbf{R}^d.
\end{equation*}
If $|x|\leq t$ and $t\geq1$, then
\begin{equation*}
  \big|D^{\alpha}I_{1}(t, x)\big|\leq ct^{-(d+|\alpha|)/2}\ \Big(1+t^{-1/2}|x|\Big)^{|\alpha|},\,\, \alpha\in\mathbf{N}^{d},\,\,  x\in\mathbf{R}^d.
 \end{equation*}
The above pointwise estimates imply the $L^1-L^{\infty}$ dispersive estimates of $e^{-itH_{0}}$:
\begin{itemize}
	\item For $0<|t| \leq 1$,
	\begin{equation}\label{freeL1Linftysmallt}
  \big\|e^{-itH_{0}}\big\|_{L^{1}(\mathbf{R}^{d})\rightarrow L^{\infty}(\mathbf{R}^{d})}\le
  	c |t|^{-d/4}.
\end{equation}
\vspace{0.1cm}
    \item For $|t|>1$,
\begin{equation}\label{freeL1Linftylarget}
  \big\|e^{-itH_{0}}\big\|_{L^{1}(\mathbf{R}^{d})\rightarrow L^{\infty}(\mathbf{R}^{d})}\le  c |t|^{-d/2}.
\end{equation}
\end{itemize}
Note that the $L^1-L^{\infty}$ estimate of free propagator $e^{-it(\Delta^2-\Delta)}$ has different time bound  for $|t|>1$ and $0<|t|\leq1$.  In contrast to the homogeneous operators, the $L^1-L^{\infty}$ estimate of propagator related to Laplacian $-\Delta$ and bi-harmonic operator $\Delta^2$,  the  time bounds of small time $0<|t|\leq 1$ and large time $|t|>1$ are the same while the inhomogeneous operator $H_{0}=\Delta^2-\Delta$ is not.

For $H=\Delta^2-\Delta+V$ with real-valued, polynomial decaying potential $V(x)$, we will establish the $L^1-L^{\infty}$ estimate in $\mathbf{R}^{3}$ for the perturbed propagator $e^{-itH}$ with the presence of zero energy obstructions. Here zero energy obstructions means there exists distributional solution to $H\psi=0$ with $\psi\in \mathcal{L}^{2}_{-1/2}(\mathbf{R}^3)$.  We provide a full classification of the zero energy obstructions: zero is a regular point, zero is a resonance and zero is an eigenvalue of $H$, see Lemma \ref{regular}-\ref{stop} below.  Follwoing the terminology of Kato in \cite{JK}, we say zero is a regular point of $H$ means zero is not an eigenvalue nor a resonance of $H$.  \emph{For $H=\Delta^2-\Delta+V$ in $\mathbf{R}^3$, we say zero is  a resonance of $H$ if there exists distributional solutions $\psi\in  \mathcal{L}^{2}_{-1/2}(\mathbf{R}^3)\setminus L^2(\mathbf{R}^3)$ such that $H\psi=0$.} Here $ \mathcal{L}^{2}_{-1/2}(\mathbf{R}^3)= \cap_{s<-\frac{1}{2}}L^{2}_{s}(\mathbf{R}^3)$ which contains $L^{2}(\mathbf{R}^3)$ as a true subset.  Actually, the property of $H=\Delta^2-\Delta+V$ at zero energy behaves similar to Schr\"odinger operator $-\Delta+V$. Recall the classification of zero energy obstructions for $-\Delta+V$ (under suitable assumptions on $V$) in $\mathbf{R}^3$, Kato in \cite{JK} disscussed zero energy as regular point, resonance and eigenvalue of $-\Delta+V$. Furthermore, the definition of zero resonance for $H$ matches Kato's zero resonance definition of $-\Delta+V$ in \cite{JK}.

Before comparing $H=\Delta^2-\Delta+V$ with $-\Delta+V$ and $\Delta^2+V$,  it is necessary to state our main theorem: the $L^{1}-L^{\infty}$ dispersive estimates of $e^{-itH}$ with zero energy obstructions.  Let $P_{ac}(H)$ denotes the projection onto the absolutely continuous spectrum space of $H$.

\begin{theorem}\label{timedecay}
Considering $H=\Delta^2-\Delta+V$ with real-valued potential satisfying $|V(x)|\lesssim(1+|x|)^{-\beta}$ for some $\beta>0$ choosing as following. Assume that $H$ has no positive embedded eigenvalues.

For $1<|t|\in\mathbf{R}$,  then:
\begin{enumerate}
	\item If $0$ is a regular point of $H$, let $\beta>3$, then
 \begin{equation*}\label{lowenergylarget}
  \big\|e^{-itH}P_{ac}(H) u\big\|_{L^{\infty}(\mathbf{R}^{3})}\lesssim\,\, |t|^{-3/2}\|u\|_{L^{1}(\mathbf{R}^{3})}.
\end{equation*}
    \item If $0$ is purely a resonance of $H$, let $\beta>5$, then
    \begin{equation*}\label{lowenergylarget}
  \big\|e^{-itH}P_{ac}(H) u-\mathcal{F}_{t}u\big\|_{L^{\infty}(\mathbf{R}^{3})}\lesssim\,\, |t|^{-3/2}\|u\|_{L^{1}(\mathbf{R}^{3})},
\end{equation*}
where $\mathcal{F}_{t}$ is a time dependent operator satisfies $\big\|\mathcal{F}_{t}\big\|_{L^{1}\rightarrow L^{\infty}}\lesssim |t|^{-1/2}$.
\vspace{0.2cm}
    \item If $0$ is purely an eigenvalue or $0$ is both resonance and eigenvalue of $H$, let $\beta>7$, then
    \begin{equation*}\label{lowenergylarget}
  \big\|e^{-itH}P_{ac}(H) u-\mathcal{G}_{t}u\big\|_{L^{\infty}(\mathbf{R}^{3})}\lesssim\,\, |t|^{-3/2}\|u\|_{L^{1}(\mathbf{R}^{3})},
\end{equation*}
where $\mathcal{G}_{t}$ is a time dependent operator satisfies $\big\|\mathcal{G}_{t}\big\|_{L^{1}\rightarrow L^{\infty}}\lesssim |t|^{-1/2}$.
\end{enumerate}

For $0<|t|\leq 1$, with $0$ in each spectral cases, let $\beta$ chosen as in $|t|>1$ respectively, then
\begin{equation*}\label{smallt}
  \big\|e^{-itH}P_{ac}(H) u\big\|_{L^{\infty}(\mathbf{R}^{3})}\lesssim\,\, |t|^{-3/4}\|u\|_{L^{1}(\mathbf{R}^{3})}.
\end{equation*}
\end{theorem}

\begin{remark}
	The presence of zero energy obstruction of $H=\Delta^2-\Delta+V$ has no effect on the $L^1-L^{\infty}$ estimate for $0<|t|\leq 1$. Indeed, for $0<|t|\leq 1$, we obtain the natural $|t|^{-3/4}$-time bound  for $H=\Delta^2-\Delta+V$ whether zero is regular or not. This can be confirmed  by  the dispersive estimate \eqref{freeL1Linftysmallt} for free propagator $e^{-itH_{0}}$.
\end{remark}
\begin{remark}
	 It is known that positive embedded eigenvalues may exist for $H=\Delta^2-\Delta+V$  even with compactly supported smooth potential. Note that $\Delta(r^{-1}e^{-br})=b^2r^{-1}e^{-br}$ for $r=|x|$ and any $b>0$.  The absence of embedded eigenvalue is a common standing assumption in dispersive equation papers.  Furthermore, we also notice that for a general selfadjoint operator $\mathcal{H}$ on $L^{2}(\mathbf{R}^d)$, if $\mathcal{H}$ has a simple embedded eigenvalue $\lambda_{0}$, Costin and Soffer in \cite{CoSo} have proved that $\mathcal{H}+\epsilon W$ can kick off the eigenvalue in a small interval around $\lambda_{0}$ under certain small perturbation of potential.
\end{remark}

The result of $L^1-L^{\infty}$ dispersive estimates for Schr\"odinger operator $-\Delta+V$ is quite rich and thoroughly. In the fundamental work \cite{JSS}, Journ\'{e}, Soffer and Sogge first proved the Schr\"odinger propagator $e^{-it(-\Delta+V)}$ satisfies
\begin{equation}\label{Schrodinger}
\big\|e^{-it(-\Delta+V)}P_{ac}(-\Delta+V)\big\|_{L^{1}(\mathbf{R}^d)\rightarrow L^{\infty}(\mathbf{R}^d)}\lesssim |t|^{-d/2},\,\,\, d\geq3,
\end{equation}
with zero is a regular point by using commutator methods.   Later, Rodnianski and Schlag in \cite{RodSchl} proved the bound \eqref{Schrodinger} for $-\Delta+V$ with rough and time-dependent potentials in $3$-dimension by making full use of the kernel of the free propagator $e^{-it\Delta}$ and the uniformly Sobolev estimates. Goldberg and Schlag in \cite{Goldberg-Schlag-1D&3D} proved \eqref{Schrodinger} for $d=1, 3$ in the regular case by using resolvent methods. In  \cite{Erdogan-Burak-Schlag-3D-2004},  Erdo\u{g}an and Schlag established the $L^1-L^{\infty}$ dispersive estimates of $-\Delta+V$ with zero energy obstructions in 3-dimension. For convenience of comparison, here we list their result briefly.  Under  suitible decaying assumption on $|V(x)|$ (see \cite[Theorem 1.1, 1.2]{Erdogan-Burak-Schlag-3D-2004}):
\begin{itemize}
	\item If there is a resonance at energy zero but zero is not an eigenvalue of $-\Delta+V$, then
	$$\big\|e^{-it(-\Delta+V)}P_{ac}(-\Delta+V)-\mathcal{F}_{t}^1\big\|_{L^{1}(\mathbf{R}^3)\rightarrow L^{\infty}(\mathbf{R}^3)}\lesssim |t|^{-3/2}.$$
	\item If there is a resonance at energy zero  and~/~or zero is an eigenvalue  of $-\Delta+V$,  then
	$$\big\|e^{-it(-\Delta+V)}P_{ac}(-\Delta+V)-\mathcal{F}_{t}^2\big\|_{L^{1}(\mathbf{R}^3)\rightarrow L^{\infty}(\mathbf{R}^3)}\lesssim |t|^{-3/2}.$$
\end{itemize}
Here $\mathcal{F}_{t}^{1}, \mathcal{F}_{t}^2$ are time-dependent, finite rank operators satisfying $\|\mathcal{F}_{j}\|_{L^1\rightarrow L^{\infty}}\lesssim |t|^{-1/2},\, \, j=1, 2$.  It is worth pointing out that Goldberg in \cite{Goldberg-GAFA-2006} proved estimate \eqref{Schrodinger} when $d=3$ for $-\Delta+V$ with almost critical potential. Comparing with the dispersive estimate of $H=\Delta^2-\Delta+V$  in Theorem \ref{timedecay},  one can conclude that $H=\Delta^2-\Delta+V$  behaves similar to Schr\"odinger operator $-\Delta+V$ in the sense of $L^1-L^{\infty}$ dispersive estimates for large time  $t>1$. More results about the dispersive estimates for the Schr\"odinger operator $-\Delta+V$, see \cite{Goldberg-Green-2, Goldberg-Green-1, Erdogan-Burak-Green-2D-2013, Fernando-Claudio-Georgi-AA-2011, Schlag, Erdogan-Burak-Schlag-3Dmatrix-2006,  GV, Schlag-CMP, Yajima-CMP-2005, Yajima-JMSJ-95} and references therein.

Recently,  Erdo\u{g}an, Green and Toprak in \cite{Burak-Green-Toprak-3D} established the $L^1- L^{\infty}$ dispersive estimates for $\Delta^2+V$ with zero energy obstruction in 3-dimension. In order to compare $H=\Delta^2-\Delta+V$ with $\Delta^2+V$, we state the results in  \cite{Burak-Green-Toprak-3D} briefly.  In \cite{Burak-Green-Toprak-3D},  under the assumptions that $\Delta^2+V$ has no positive embedded eigenvalue and suitible decaying assumption on $V(x)$, they proved that (see \cite[Theorem 1.1]{Burak-Green-Toprak-3D})
\begin{itemize}
	\item For $0<|t|\leq1$, then
	$$\big\|e^{-it(\Delta^2+V)}P_{ac}(\Delta^2+V)\big\|_{L^{1}(\mathbf{R}^3)\rightarrow L^{\infty}(\mathbf{R}^3)}\lesssim~ |t|^{-3/4}.$$
	\item For $|t|>1$, if zero is regular or there is a first kind resonance at zero of $\Delta^2+V$, then
	$$\big\|e^{-it(\Delta^2+V)}P_{ac}(\Delta^2+V)\big\|_{L^{1}(\mathbf{R}^3)\rightarrow L^{\infty}(\mathbf{R}^3)}\lesssim~ |t|^{-3/4}.$$
	\item For $|t|>1$, if there is a second kind or third kind resonance at zero of $\Delta^2+V$,  then there exist time dependent operators $\mathcal{F}_{t}^{3}$ and $\mathcal{F}_{t}^{4}$ respectively, such that	$$\big\|e^{-it(\Delta^2+V)}P_{ac}(\Delta^2+V)-\mathcal{F}_{t}^{j}\big\|_{L^{1}(\mathbf{R}^3)\rightarrow L^{\infty}(\mathbf{R}^3)}\lesssim |t|^{-3/4},\,\, j=3, 4,$$
	where $\mathcal{F}_{t}^{3}, \mathcal{F}_{t}^4$  satisfy $\|\mathcal{F}_{t}^{j}\|_{L^1\rightarrow L^{\infty}}\lesssim |t|^{-1/4},\, \, j=3, 4$.
\end{itemize}
Notice that the third kind resonance at zero of $\Delta^2+V$  is actually zero eigenvalue according to the definition and classification of threshold subspaces in \cite{Burak-Green-Toprak-3D}.

Comparing to our main Theorem \ref{timedecay}, for $0<|t|\leq 1$, $H=\Delta^2-\Delta+V$ and $\Delta^2+V$ satisfy the same $|t|^{-3/4}$-time bound of the  $L^1-L^{\infty}$ dispersive estimate. However when $|t|>1$ and both in the regular case, we obtain $|t|^{-3/2}$-time decay for $H=\Delta^2-\Delta+V$ while  $\Delta^2+V$  satisfies $|t|^{-3/4}$-time decay. Further,  when $|t|>1$ and both in the zero eigenvalue case, we obtain $|t|^{-1/2}$-time decay for $H=\Delta^2-\Delta+V$ while  $\Delta^2+V$ satisfies $|t|^{-1/4}$-time decay.   In addition, the  $L^1-L^{\infty}$ dispersive estimates of $\Delta^2+V$ in 4-dimension,  please see \cite{Green-Toprak-4D}. For the remaining dimenisonal cases, up to my knowledge, it is still open so far.


As usual, we use spectrum representation theorem to write
\begin{equation*}
e^{-itH}P_{ac}(H)=\frac{1}{2\pi i}\int_{0}^{\infty}e^{-it\lambda}\Big[R_{V}^{+}(\lambda)-R_{V}^{-}(\lambda)\Big]d\lambda,
\end{equation*}
Here, the difference of the perturbed resolvents provides the spectral measure by Stone's formula.   Let $\lambda\in\mathbf{R}^{+}$, we define the limiting resolvent operators by
\begin{align*}
&R^{\pm}_{0}(\lambda):=R_{0}(\lambda\pm i0)=\lim_{\epsilon\downarrow0}\Big(H_{0}-(\lambda\pm i\epsilon)\Big)^{-1};\\
&R_{V}^{\pm}(\lambda):=R_{V}(\lambda\pm i0)=\lim_{\epsilon\downarrow0}\Big(H-(\lambda\pm i\epsilon)\Big)^{-1}.
\end{align*}

For free resolvent $R_{0}(z):=(H_{0}-z)^{-1}$, we have the following splitting identity:
\begin{equation}\label{freeresolventidentity}
  R_{0}(z)=\frac{1}{2\sqrt{1/4+z}}\Bigg[\Big(-\Delta+\frac{1}{2}-\sqrt{1/4+z}\Big)^{-1}-\Big(-\Delta+\frac{1}{2}+\sqrt{1/4+z}\Big)^{-1}\Bigg].
\end{equation}
Here $R_{\Delta}(z)$ denotes the resolvent of $-\Delta$. Since $H_{0}=\Delta^2-\Delta$ is essentially selfadjoint and $\sigma_{ac}(H_{0})=[0, \infty)$, by Weyl's criterion $\sigma_{ess}(H)=[0, \infty)$ for a sufficiently decaying potential.
Note that, by basic calculation, for $z\in\mathbf{C}\setminus[0, \infty)$ with $0<\arg(z)<2\pi$ we have ${\rm Im}~ \big(\sqrt{1/4+z}-1/2\big)^{1/2}>0$. Using  identity \eqref{freeresolventidentity}, for $\lambda>0$ we have
\begin{equation}\label{freelimiting}
	R_{0}^{\pm}(\lambda)=\frac{1}{2\sqrt{1/4+\lambda}}\Bigg[R_{\Delta}^{\pm}\Big(\sqrt{1/4+\lambda}-\frac{1}{2}\Big)-R_{\Delta}\Big(-\frac{1}{2}-\sqrt{1/4+\lambda}\Big)\Bigg].
\end{equation}
Note that $R_{\Delta}\Big(-\frac{1}{2}-\sqrt{1/4+\lambda}\Big)\in B(L^2, L^2)$ since $-\Delta$ has nonnegative spectrum. Further, by the limiting absorption principle (see \cite{Agmon, BenArtzi-Devinatz-1987}), $R_{\Delta}^{\pm}\Big(\sqrt{1/4+\lambda}-\frac{1}{2}\Big)$ is well-defined between weighted $L^2$ spaces. Therefore, $R_{0}^{\pm}(\lambda)$ is also well-defined in weighted $L^2$ spaces. For the perturbed resolvent $R_{V}^{\pm}(\lambda)$,   we will extend this property to $R_{V}^{\pm}(\lambda)$ in Section 5.

In the literature, in order to obtain the asymptotic propertities of the spectral measure of $H=\Delta^2-\Delta+V$ when $\lambda$ close to $0$ and $\infty$, we need to derive the low energy asymptotic expansion of $R^{\pm}_{V}(\lambda)$ when $\lambda\rightarrow0$ and the high energy decay estimate for  $R^{\pm}_{V}(\lambda)$ when $\lambda\rightarrow\infty$. For free resolvent $R_{0}^{\pm}(\lambda)$, we apply identity \eqref{freelimiting} to get the low energy asymptotic expansion and high energy decay estimate by using the known results of Laplacian , see \cite{JK}.  For Schr\"odinger operator $-\Delta+V$, Kato, Jensen and Nenciu's series work \cite{JK, J, J1, JN, JN2} has established quite a standard approach to obtain the  low energy asymptotic expansion and high energy decay estimates. For general operator $P(D)+V$, Murata in \cite{MM} already gave the asymptotic expansion  at threshold point for a class of $P(D)+V$ which includes $H=\Delta^2-\Delta+V$.  Here we do not follow Murata's result in \cite{MM} since Jensen and Nenciu's processes in \cite{JN} is more directly and succinctly.

For perturbed resolvent $R_{V}^{\pm}(\lambda)$, we apply the symmetric resolvent identity:
\begin{equation}
R^{\pm}_{V}(\lambda)=R^{\pm}_{0}(\lambda)-R^{\pm}_{0}(\lambda)v\Big(M^{\pm}(\lambda)\Big)^{-1}vR^{\pm}_{0}(\lambda)
\end{equation}
where $M^{\pm}(\lambda)=U+vR^{\pm}_{0}(\lambda)v$, $v(x)=|V(x)|^{1/2}$ and
\begin{equation*}
U(x)=\begin{cases}
1,\,\, &\,\, V(x)\geq0;\\
-1,\,\, &\,\, V(x)<0.
\end{cases}
\end{equation*}
Since in the low frequency portion of $H=\Delta^2-\Delta+V$, the major operator is $-\Delta+V$. Thus Jensen and Nenciu's approach should works. Hence making use of the approach established in \cite{JN}, we obtain the low energy asymptotic expansion of $R_{V}^{\pm}(\lambda)$ with the presence of zero energy resonance or zero eigenvalue of $H=\Delta^2-\Delta+V$. Further, we identified the zero energy resonance space.

The paper is organized as follows. In Section \ref{free-evolution}, we show the natural dispersive bound for the free propagator. In Section \ref{asymptotic-expansion}, we derive  expansions for the perturbed resolvent around the threshold with the presence of zero energy obstruction and  identified the zero energy resonance space. In Section \ref{lowenergy}, we utilize these expansions to prove the low energy part dispersive estimates in Theorem \ref{timedecay}. In the last section, we develop the high energy decay estimates and the limiting absorption principle for the perturbed resolvent. Then apply the high energy decay estimates to prove the high energy part dispersive estimates in Theorem \ref{timedecay}.


\section{The free evolution}\label{free-evolution}
In this section we obtain expansions for the free resolvent operators $R^{\pm}_{0}(\lambda)$ using identity \eqref{freeresolventidentity} and the representation of the free Schr\"odinger  resolvent.  Using resolvent expansions, we  establish dispersive estimates for the free  evolution $e^{-itH_{0}}$.

Before deducing the expansions, we first introduce some notations for the convenience reading.  We define the weighted $L^2$ spaces:
\begin{equation*}
L^{2}_{s}(\mathbf{R}^3):=\Big\{~f:~ (1+|\cdot|)^{s}f\in L^{2}(\mathbf{R}^3)~\Big\},\,\, s\in\mathbf{R}.
\end{equation*}
For any $s, s'\in\mathbf{R}$, $B(s, s')$ denotes the family of bounded linear operators from $L^{2}_{s}(\mathbf{R}^3)$ to $L^{2}_{s'}(\mathbf{R}^3)$. For an operator $\mathcal{E}(\lambda)$, we write $\mathcal{E}(\lambda)=O_{1}(\lambda^{-a})$ if its kernel $\mathcal{E}(\lambda; x, y)$ satisfies:
\begin{equation*}
\sup_{x, y\in\mathbf{R}^3, \lambda>0}\Big[\lambda^{a}\big|\mathcal{E}(\lambda; x, y)\big|+\lambda^{a+1}\big|\partial_{\lambda}\mathcal{E}(\lambda; x, y)\big|\Big]<\infty.
\end{equation*}
Similarly, we use the notation $\mathcal{E}(\lambda)=O_{1}(\lambda^{-a}g(x, y))$ if $\mathcal{E}(\lambda; x, y)$ satisfies
\begin{equation*}
\big|\mathcal{E}(\lambda; x, y)\big|+\lambda\big|\partial_{\lambda}\mathcal{E}(\lambda; x, y)\big|\lesssim \lambda^{a}g(x, y).
\end{equation*}

Next, we show the process of deducing the asymtotic expansions. Recall the expression of the free Schr\"odinger resolvents in 3-dimension (see e.g. \cite{JK}),
\begin{equation}
	R_{\Delta}^{\pm}(\eta^2; x, y)=\frac{e^{\pm i\eta|x-y|}}{4\pi|x-y|},
\end{equation}
where $\eta=\Big(\sqrt{1/4+\lambda}-\frac{1}{2}\Big)^{1/2}$. Therefore, by \eqref{freelimiting},
\begin{equation}\label{freekernel}
	R^{\pm}_{0}(\lambda; x, y)=\frac{1}{1+2\eta^2}\Bigg(\frac{e^{\pm i\eta|x-y|}}{4\pi|x-y|}-\frac{e^{-\sqrt{1+\eta^2}|x-y|}}{4\pi|x-y|}\Bigg).
\end{equation}

\begin{proposition}\label{freedispersive}
	For the free evolution $e^{-itH_{0}}$, denotes $\mathcal{R}_{0}(\eta; x, y)=R^{+}_{0}(\lambda; x, y)-R^{-}_{0}(\lambda; x, y)$. Then we have the following uniformly bounds:
	
	For $1<t\in\mathbf{R}$,  we have
	\begin{equation}\label{freelowdispersivelarget}
		\sup_{x, y\in\mathbf{R}^3}\Bigg|\int_{0}^{\infty}e^{-it(\eta^4+\eta^2)}\mathcal{R}_{0}(\eta; x, y)~(4\eta^3+2\eta)~d\eta\Bigg|\lesssim\,\, |t|^{-3/2}.
	\end{equation}
	
	For $0<t\leq 1$, we have
	\begin{equation}\label{freedispersivesmallt}
		\sup_{x, y\in\mathbf{R}^3}\Bigg|\int_{0}^{\infty}e^{-it(\eta^4+\eta^2)}\mathcal{R}_{0}(\eta; x, y)~(4\eta^3+2\eta)~d\eta\Bigg|\lesssim\,\, |t|^{-3/4}.
	\end{equation}
\end{proposition}

\begin{proof}
	Note that, for the difference of the free resolvents we have
	\begin{equation}\label{differetfree}
		\big|\mathcal{R}_{0}(\eta; x, y)\big|=\frac{\eta}{1+2\eta^2}\Bigg|\frac{e^{i\eta|x-y|}-1+1-e^{-i\eta|x-y|}}{4\pi\eta|x-y|}\Bigg|\lesssim \frac{\eta}{1+2\eta^2}
	\end{equation}
	uniformly in $x, y$ by the mean value theorem. Further,
	\begin{equation}\label{derivativefree}
		\Big|\frac{d}{d\eta}\mathcal{R}_{0}(\eta; x, y)\Big|\leq\frac{1}{1+2\eta^2}\Bigg|\frac{e^{i\eta|x-y|}+e^{-i\eta|x-y|}}{4\pi}\Bigg|+\frac{4\eta}{(1+2\eta^2)^2}\Bigg|\frac{e^{i\eta|x-y|}-e^{-i\eta|x-y|}}{4\pi|x-y|}\Bigg|\lesssim \frac{1}{1+2\eta^2}.
	\end{equation}
	For  large time $t>1$ case, we have
	\begin{equation*}
		\begin{split}
			&~\Bigg|\int_{0}^{\infty}e^{-it(\eta^4+\eta^2)}\mathcal{R}_{0}(\eta; x, y)~(4\eta^3+2\eta)~d\eta\Bigg|\\
			\lesssim &~\frac{1}{|t|}~\Bigg|e^{-it(\eta^4+\eta^2)}\mathcal{R}_{0}(\eta; x, y)\Big|_{0}^{\infty}\Bigg|+\frac{1}{|t|}~\Bigg|\int_{0}^{\infty}e^{-it(\eta^4+\eta^2)}\frac{d}{d\eta}\mathcal{R}_{0}(\eta; x, y)d\eta\Bigg|\\
			\lesssim &~\frac{1}{|t|} \Bigg|\int_{0}^{\infty}e^{-it(\eta^4+\eta^2)}\frac{e^{i\eta|x-y|}+e^{-i\eta|x-y|}}{1+2\eta^2}d\eta\Bigg|+\frac{1}{|t|}\Bigg|\int_{0}^{\infty}e^{-it(\eta^4+\eta^2)}\frac{4\eta^2}{(1+2\eta^2)^2}\frac{e^{i\eta|x-y|}-e^{-i\eta|x-y|}}{4\pi|x-y|}d\eta\Bigg|\\
			\lesssim &~ |t|^{-3/2}, \,\, t>1.
		\end{split}
	\end{equation*}
In the last inequality we use the van der Corput's lemma, see e.g. \cite{Stein}. 	

For small time $0<t\leq 1$, we have
\begin{equation*}
		\int_{0}^{\infty}e^{-it(\eta^4+\eta^2)}\mathcal{R}_{0}(\eta; x, y)~(4\eta^3+2\eta)~d\eta=I_{1}(t; x, y)+I_{2}(t; x, y)
\end{equation*}
where
\begin{align*}
	& I_{1}(t; x, y)=\int_{0}^{t^{-1/4}}e^{-it(\eta^4+\eta^2)}\mathcal{R}_{0}(\eta; x, y)~(4\eta^3+2\eta)~d\eta;\\
	& I_{2}(t; x, y)=\int_{t^{-1/4}}^{\infty}e^{-it(\eta^4+\eta^2)}\mathcal{R}_{0}(\eta; x, y)~(4\eta^3+2\eta)~d\eta.
\end{align*}
For $I_{1}(t; x, y)$ ,  we have
\begin{equation*}
	\Big|I_{1}(t; x, y)\Big|\lesssim \int_{0}^{t^{-1/4}}\Big|\mathcal{R}_{0}(\eta; x, y)(4\eta^3+2\eta)\Big|d\eta\lesssim |t|^{-3/4},\,\, 0<t\leq1.
\end{equation*}	
For $I_{2}(t; x, y)$, integral by parts, we have
\begin{equation*}
	\begin{split}
		\Big|I_{2}(t; x, y)\Big|\lesssim &~ \frac{1}{|t|}\bigg|e^{-it(\eta^4+\eta^2)}\mathcal{R}_{0}(\eta; x, y)\Big|_{t^{-1/4}}^{\infty}\bigg|+\frac{1}{|t|}\int_{t^{-1/4}}^{\infty}\Big|\frac{d}{d\eta}\mathcal{R}_{0}(\eta; x, y)\Big|d\eta\\
		\lesssim &~ \frac{|t|^{-3/4}}{|t|^{1/2}+2}+\frac{1}{|t|(1+|t|^{-1/4})}\lesssim |t|^{-3/4},\,\, 0<t\leq 1.
	\end{split}
\end{equation*}
\end{proof}

\section{The asymptotic expansions of resolvent $R_{V}^{\pm}(\lambda)$}\label{asymptotic-expansion}
\setcounter{equation}{0}

In this section, we aim to obtain the low energy asymptotic expansion of $R_{V}^{\pm}(\lambda)$.  For $\lambda$ near the only threshold point zero, we use the symmetric resolvent identity:
\begin{equation}\label{symmetriresolventidentity}
    R^{\pm}_{V}(\lambda)=R^{\pm}_{0}(\lambda)-R^{\pm}_{0}(\lambda)v\Big(M^{\pm}(\lambda)\Big)^{-1}vR^{\pm}_{0}(\lambda)
\end{equation}
where $M^{\pm}(\lambda)=U+vR^{\pm}_{0}(\lambda)v$, $v(x)=|V(x)|^{1/2}$ and
\begin{equation}
	U(x)=\begin{cases}
		1,\,\, &\,\, V(x)\geq0;\\
		-1,\,\, &\,\, V(x)<0.
	\end{cases}
\end{equation}

\subsection{Asymptotic expansion of $R_{0}^{\pm}(\lambda)$ near $\lambda=0$}

Recall the kernel of $R_{0}^{\pm}(\lambda)$:
\begin{equation*}
   R_{0}^{\pm}(\lambda; x, y)=\frac{1}{1+2\eta^2}\Bigg[\frac{e^{\pm i\eta|x-y|}}{4\pi|x-y|}-\frac{e^{-\sqrt{1+\eta^2}|x-y|}}{4\pi|x-y|}\Bigg],
\end{equation*}
where $\eta=\big(\sqrt{1/4+\lambda}-1/2\big)^{1/2}$. Expand each term into Taylor series at $\eta=0$, we have:
\begin{lemma}\label{freeresolventexpansion}
For $\lambda>0$ and $\eta=\big(\sqrt{1/4+\lambda}-1/2\big)^{1/2}$, we have the following formally expansions of  $R_{0}^{\pm}(\lambda)$ as $\lambda\downarrow 0$:
\begin{equation}\label{lap odd}
   R_{0}^{\pm}(\lambda)=G_{0}\pm i\eta G_{1}+\eta^2 G_{2}\pm i\eta^3 G_{3}+\eta^4 G_{4}+O_{1}(\eta^5|x-y|^4),
\end{equation}
where $G_{j}, j=0, 1, 2, 3, 4$ are operators given by the following kernels
\begin{equation*}
  \begin{split}
  & G_{0}(x, y)=\frac{1}{4\pi|x-y|}-\frac{e^{-|x-y|}}{4\pi|x-y|};\,\, G_{1}(x, y)=\frac{1}{4\pi};\,\, G_{3}(x, y)=\frac{-1}{2\pi}+\frac{-|x-y|^2}{24\pi};\\
  & G_{2}(x, y)=\Bigg(\frac{-|x-y|}{8\pi}+\frac{e^{-|x-y|}}{8\pi}\Bigg)-\Bigg(\frac{1}{2\pi|x-y|}-\frac{e^{-|x-y|}}{2\pi|x-y|}\Bigg);\\
  & G_{4}(x, y)=\frac{1-e^{-|x-y|}}{\pi|x-y|}+\frac{|x-y|-e^{-|x-y|}}{4\pi}+\Bigg(\frac{|x-y|^3}{96\pi}-\frac{1+|x-y|}{32\pi}e^{-|x-y|}\Bigg).
  \end{split}
\end{equation*}

Note that $G_{0}(x, y)$ is actually the kernel of free resolvent $R_{0}(0):=(-\Delta)^{-1}-(-\Delta+1)^{-1}$. Furthermore, $G_{0}\in B(s, -s')$ with $s, s'>1/2$ and $s+s'>2$. For $j=1, 2, 3, 4$, $G_{j}\in B(s, -s')$ with $s, s'>j+1/2$. The operator with kernel $|x-y|^4$ belongs to $B(s, -s')$ with $s, s'>11/2$.
\end{lemma}
\begin{proof}
	Note that $|x-y|^j$ with $j=0, 1, 2, 3,4$, one obtain
	\begin{equation*}
		\int_{\mathbf{R}^3}\int_{\mathbf{R}^3}\Big(1+|x|\Big)^{-2s'}\big|x-y\big|^{2j}\Big(1+|y|\Big)^{-2s}dxdy<\infty
	\end{equation*}
	for $s, s'>j+3/2$. Since $\Big(4\pi|x-y|\Big)^{-1}$ is the kernel of $(-\Delta)^{-1}$,  we know $(-\Delta)^{-1}\in B(s, -s')$ and $(-\Delta)^{-1}$ is compact in $B(s, -s')$ with $s+s'>2$ by \cite[Lemma 2.3]{J}.
\end{proof}

\subsection{Asymptotic expansion of $R_{V}^{\pm}(\lambda)$ near $\lambda=0$}
In order to obtain the asymptotic expansions of $R_{V}^{\pm}(\lambda)$ near zero threshold, we need to derive the expansions for $\big(M^{\pm}(\lambda)\big)^{-1}$ by identity \eqref{symmetriresolventidentity}. The behavior of $\big(M^{\pm}(\lambda)\big)^{-1}$ as $\lambda\rightarrow0$ depends on the type of resonances at zero energy, see Definition \ref{resonace} below. We get these expansions case by case and establish their contribution to the spectral measure in Stone's formula. From the expansions of free resolvent $R_{0}^{\pm}(\lambda)$ in the weighted spaces $B(s, -s')$, see Lemma \ref{freeresolventexpansion}, we have the following expansions for $M^{\pm}(\lambda)$.

\begin{lemma}
	Let $P=v\langle\cdot, v\rangle\|V\|_{L^{1}(\mathbf{R}^3)}^{-1}$ denotes the orthogonal projection onto the span of $v$. Assume that $v(x)\lesssim (1+|x|)^{-\beta/2}$ with some $\beta>9$. Then for $0<\lambda<1$ in $B(0, 0)$, we have
	\begin{equation}
		M^{\pm}(\lambda)=T_{0}+i\frac{\|V\|_{L^1}}{4\pi}\eta P+\eta^2 vG_{2}v+i\eta^3vG_{3}v+O_{1}\Big(\eta^4v(x)|x-y|^3v(y)\Big)
	\end{equation}
	where $T_{0}=U+vG_{0}v$.
	\end{lemma}
	\begin{proof}
		From Lemma \ref{freeresolventexpansion}, we need only to show $\big|G_{4}(x, y)\big|\lesssim |x-y|^3$. Since
		\begin{equation*}
			\Bigg|\frac{1-e^{-|x-y|}}{\pi|x-y|}\Bigg|\lesssim \begin{cases}
				1,\,\,&\,\, |x-y|<1;\\ |x-y|^{-1},\,\, &\,\, |x-y|\geq1,
			\end{cases}
		\end{equation*}
		thus the lemma holds by the representation of $G_{4}(x, y)$.
	\end{proof}
In order to get the asymptotic expansions of $\Big(M^{\pm}(\lambda)\Big)^{-1}$ at $\lambda=0$, we deal with zero in three cases: regular(not resonance nor eigenvalue), resonance and eigenvalue.
\begin{definition}\label{resonace}
i) If $T_{0}=U+vG_{0}v$ is invertible on $L^{2}(\mathbf{R}^3)$, we say zero is a regular point of $H=\Delta^2-\Delta+V$. \quad ii) If $T_{0}$ is not invertible and $T_{1}=S_{1}PS_{1}$ is invertible on $S_{1}L^{2}(\mathbf{R}^3)$, we say zero is a resonance of $H$. Here $S_{1}$ is the Riesz projection onto $\ker(T_{0})$. \quad iii) If $T_{1}$ is not invertible and $T_{2}=S_{2}vG_{2}vS_{2}$ is invertible on $S_{2}L^{2}(\mathbf{R}^3)$, we say zero is an eigenvalue of $H$. Here $S_{2}$ is the Riesz projection onto $\ker(T_{1})$.
\end{definition}
\begin{remark}\label{propertyofS}
	i) Note that $S_{2}\leq S_{1}$ and $S_{1}$ is of finite rank. Since $vG_{0}v$ is a compact operator in $L^2(\mathbf{R}^3)$ by the proof of Lemma \ref{freeresolventexpansion}, thus $T_{0}$ is a compact perturbation of $U$. Hence, the Fredholm alternative theorem guarantees that $S_{1}$ is a finite-rank projection.\quad ii) $PS_{2}=S_{2}P=0$.\quad iii) If $0\neq S_{2}=S_{1}$, then zero is both eigenvalue and resonance of $H=\Delta^2-\Delta+V$. Otherwise,  zero is purely an eigenvalue of $H$ provided $0\neq S_{2}<S_{1}$.\quad iv) Denotes $D_{0}=\Big(T_{0}+S_{1}\Big)^{-1}$ and $D_{1}=\Big(T_{1}+S_{2}\Big)^{-1}$,  then $S_{1}D_{0}=D_{0}S_{1}=S_{1}$ and $S_{2}D_{1}=D_{1}S_{2}=S_{2}$.
\end{remark}

For better understanding zero threshold, we proceed to establish the relationship between the spectral subspaces $S_{1}L^2(\mathbf{R}^3),\,\, S_{2}L^2(\mathbf{R}^3)$ and distributional solutions to $H\psi=0$.  For any $s_{0}\in\mathbf{R}$, denotes $\mathcal{L}^{2}_{s_{0}}(\mathbf{R}^3)=\cap_{s<s_{0}}L^{2}_{s}(\mathbf{R}^3)$. Especially, $L^{2}(\mathbf{R}^3)\subsetneq\mathcal{L}^{2}_{0}(\mathbf{R}^3)$.

\begin{lemma}\label{regular}
	Assume $v(x)\lesssim (1+|x|)^{-s}$ with some $s>3/2$. If $\phi\in S_{1}L^{2}(\mathbf{R}^3)\setminus\{0\}$, then $\phi=Uv\psi$ where $\psi\in\mathcal{L}^{2}_{-1/2}(\mathbf{R}^3)$ satisfies $H\psi=0$ in the distributional sense, and
	\begin{equation*}
		\psi(x)=-G_{0}v\phi=-\int_{\mathbf{R}^3}\Bigg(\frac{1}{4\pi|x-y|}-\frac{e^{-|x-y|}}{4\pi|x-y|}\Bigg)v(y)\phi(y)dy.
	\end{equation*}
	Conversely, if $\psi\in\mathcal{L}^{2}_{-1/2}(\mathbf{R}^3)$ satisfies $H\psi=0$, then $\phi=Uv\psi\in S_{1}L^{2}(\mathbf{R}^3)$.
\end{lemma}
\begin{proof}
	For $\phi\in S_{1}L^2$, then $(U+vG_{0}v)\phi=0$ which implies $\phi=Uv(-G_{0}v\phi)=Uv\psi$. Since $\psi=-G_{0}v\phi$, then
	\begin{equation*}
		|\psi|\lesssim \int_{\mathbf{R}^3}\frac{1}{4\pi|x-y|}\big|v(y)\phi(y)\big|dy=(-\Delta)^{-1}\Big(|v\phi|\Big).
	\end{equation*}
	Since $(-\Delta)^{-1}\in B(s, -s')$ for $s, s'>1/2$ and $s+s'>2$ by \cite[Lemma 2.3]{J}, then  $\psi\in\mathcal{L}^{2}_{-1/2}(\mathbf{R}^3)$.
	 Recall that $G_{0}=(-\Delta)^{-1}-(-\Delta+1)^{-1}$. For any $\varphi\in C_{0}^{\infty}(\mathbf{R}^3)$, then
	\begin{equation*}
		\langle H\psi, \varphi\rangle=\langle -\phi, vG_{0}H\varphi \rangle=\langle -\phi, v\varphi+vG_{0}V\varphi \rangle=\langle-v\phi+vUv\psi, \varphi\rangle=0.
	\end{equation*}
	
	If $\psi\in\mathcal{L}^{2}_{-1/2}(\mathbf{R}^3)$ satisfies $H\psi=0$, we show $(U+vG_{0}v)\phi=0$. Since $0=H\psi=(\Delta^2-\Delta)\psi+vUv\psi=(\Delta^2-\Delta)\psi+v\phi$, thus $\psi=-(\Delta^2-\Delta)^{-1}v\phi$. Hence $U\phi+vG_{0}v\phi=UUv\psi-v\psi=0.$
\end{proof}

\begin{lemma}\label{S_2}
	Assume $v(x)\lesssim(1+|x|)^{-s}$ with some $s>5/2$. If $\phi\in S_{2}L^{2}(\mathbf{R}^3)\setminus\{0\}$, then $\phi=Uv\psi$ where $\psi\in L^{2}(\mathbf{R}^3)$ satisfies $H\psi=0$ in the distributional sense. Conversely, if $\psi\in L^{2}(\mathbf{R}^3)$ satisfies $H\psi=0$, then $\phi=Uv\psi\in S_{2}L^{2}(\mathbf{R}^3)$.
\end{lemma}
\begin{proof}
	Since $\phi\in S_{2}L^{2}(\mathbf{R}^3)$, then $0=\langle S_{1}PS_{1}\phi, \phi\rangle=\langle P\phi, P\phi\rangle$. Thus $P\phi=0$. Since $S_{2}\leq S_{1}$, thus  $H\psi=0$ and
	\begin{equation*}
		\psi(x)=-\int_{\mathbf{R}^3}\frac{1}{4\pi|x-y|}v(y)\phi(y)dy+\int_{\mathbf{R}^3}\frac{e^{-|x-y|}}{4\pi|x-y|}v(y)\phi(y)dy:=\psi_{1}(x)+\psi_{2}(x).
	\end{equation*}
	Note that $\psi_{2}(x)=(-\Delta+1)^{-1}(v\phi)\in L^2(\mathbf{R}^3)$, since $(-\Delta+1)^{-1}\in B(0, 0)$ and $v\phi\in L^2(\mathbf{R}^3)$. Since $P\phi=0$, i.e. $\int_{\mathbf{R}^3} v(y)\phi(y)dy=0$, then
	\begin{equation*}
		\begin{split}
			\psi_{1}(x)= &\,\, \int_{\mathbf{R}^3} \frac{-1}{4\pi|x-y|}v(y)\phi(y)dy+\int_{\mathbf{R}^3} \frac{1}{4\pi(1+|x|)}v(y)\phi(y)dy\\
			= &\,\, \frac{-1}{4\pi}\int_{\mathbf{R}^3} \frac{1+|x|-|x-y|}{(1+|x|)|x-y|}v(y)\phi(y)dy\\
			\leq &\,\,\frac{1}{1+|x|}\int_{\mathbf{R}^3} \frac{1+|y|}{4\pi|x-y|}\big|v(y)\phi(y)\big|dy\\
			= &\,\, (1+|x|)^{-1}(-\Delta)^{-1}\Big((1+|\cdot|)v\phi\Big).
		\end{split}
	\end{equation*}
	Since $(1+|\cdot|)v\phi\in L^{2}_{s+1}(\mathbf{R}^3)$, then $(-\Delta)^{-1}\Big((1+|\cdot|)v\phi\Big)\in\mathcal{L}^{2}_{-1/2}(\mathbf{R}^3)$ by the boundedness of $(-\Delta)^{-1}$ in $B(s, -s')$, see \cite[Lemma 2.3]{J}. Thus $\psi_{1}\in L^{2}_{1/2-\epsilon}(\mathbf{R}^3)\subset L^{2}(\mathbf{R}^3)$ for any tiny $\epsilon>0$.
	
	If $\psi\in L^{2}(\mathbf{R}^3)$ satisfies $H\psi=0$, we show $P\phi=0$. Since $\psi\in L^{2}(\mathbf{R}^3)\subset \mathcal{L}^{2}_{-1/2}(\mathbf{R}^3)$, thus we have $L^{2}(\mathbf{R}^3)\ni\psi=\psi_{1}+\psi_{2}$. Since $\psi_{2}=(-\Delta+1)^{-1}\Big(v\phi\Big)\in L^2(\mathbf{R}^3)$, thus $\psi_{1}\in L^2(\mathbf{R}^3)$. However,
	\begin{align*}
		\psi_{1}(x)=\frac{-1}{4\pi}\int_{\mathbf{R}^3} \frac{1+|x|-|x-y|}{(1+|x|)|x-y|}v(y)\phi(y)dy-\int_{\mathbf{R}^3} \frac{1}{4\pi(1+|x|)}v(y)\phi(y)dy
	\end{align*}
	and
	\begin{equation*}
	\Big|\frac{-1}{4\pi}\int_{\mathbf{R}^3} \frac{1+|x|-|x-y|}{(1+|x|)|x-y|}v(y)\phi(y)dy\Big|\leq(1+|x|)^{-1}(-\Delta)^{-1}\Big((1+|\cdot|)v\phi\Big)\in L^{2}(\mathbf{R}^3)
	\end{equation*}
	which implies  $\frac{1}{4\pi(1+|x|)}\int_{\mathbf{R}^3} v(y)\phi(y)dy\in L^{2}_{x}(\mathbf{R}^3)$. Hence $\int_{\mathbf{R}^3} v(y)\phi(y)dy=0$.
\end{proof}

\begin{lemma}\label{stop}
	Assume $v(x)\lesssim (1+|x|)^{-s}$ with some $s>7/2$, then $\ker(S_{2}vG_{2}vS_{2})=\{0\}$.
\end{lemma}
\begin{proof}
	For $\phi\in\ker(S_{2}vG_{2}vS_{2})$, then $\phi\in S_{2}L^2(\mathbf{R}^3)$, thus $P\phi=0$ which implies $vG_{1}v\phi=0$ by Lemma \ref{S_2} and the definition of $P$.  Let $\zeta=\big(\sqrt{1/4+z}-1/2\big)^{1/2}$. Therefore, by \eqref{freeresolventidentity} we have
	\begin{equation*}
		\begin{split}
			0=&\,\,\big\langle S_{2}vG_{2}vS_{2}\phi,\,\,\phi\big\rangle=\big\langle G_{2}v\phi,\,\,v\phi\big\rangle\\
			=&\,\, \lim_{\zeta\rightarrow0}\Bigg\langle\frac{R_{0}(z)-G_{0}-iG_{1}\zeta}{\zeta^2}v\phi,\,\,v\phi\Bigg\rangle
			= \lim_{\zeta\rightarrow0}\Bigg\langle\frac{R_{0}(\zeta^4+\zeta^2)-G_{0}}{\zeta^2}v\phi,\,\,v\phi\Bigg\rangle \\
			=&\,\, \lim_{\zeta\rightarrow0}\frac{1}{\zeta^2}\Bigg\langle\Bigg(\frac{1}{\xi^4+\xi^2-(\zeta^4+\zeta^2)}-\frac{1}{\xi^4+\xi^2}\Bigg)\widehat{v\phi}(\xi),\,\,\widehat{v\phi}(\xi)\Bigg\rangle\\
			=&\,\, \int_{\mathbf{R}^3}\frac{|\widehat{v\phi}(\xi)|^2}{(\xi^4+\xi^2)^2}d\xi.
		\end{split}
	\end{equation*}
	Here we used the dominated convergence theorem as $\zeta\rightarrow0$ with ${\rm Re}~(\zeta^4+\zeta^2)<0$ (by choose $0<|z|<1$ with ${\rm Re}~(z)<0$) in the last identity. Hence we have $v\phi=0$ since $v\phi\in L^1$. Note that $\phi\in S_{2}L^2\subset S_{1}L^2$, then  $\phi=Uv(-G_{0}v\phi)=0$ by Lemma \ref{S_2}.	
\end{proof}

In the rest of this section, we aim to obtain suitable expansions for $\Big(M^{\pm}(\lambda)\Big)^{-1}$ as $\lambda\rightarrow0$ in the three cases:  zero is a regular point, zero is a resonance and zero is an eigenvalue. Recall that $\eta=(\sqrt{1/4+\lambda}-1/2)^{1/2}$. Thus $\lambda\rightarrow0$ equals $\eta\rightarrow0$.
\begin{theorem}\label{Mexpansions}
	i) If zero is a regular point of $H=\Delta^2-\Delta+V$ with $|V(x)|\lesssim(1+|x|)^{-\beta}$ for some $\beta>3$, then
	\begin{equation*}
		\Big(M^{\pm}(\lambda)\Big)^{-1}=T_{0}^{-1}\mp i\frac{\|V\|_{L^1}}{4\pi}T_{0}^{-1}PT_{0}^{-1}\eta+O_{1}(\eta^2)
	\end{equation*}
	in $B(0, 0)$ as $\eta\rightarrow0$.
	
	ii) If zero is purely a resonance of $H=\Delta^2-\Delta+V$ with $|V(x)|\lesssim(1+|x|)^{-\beta}$ for some $\beta>5$, then
	\begin{equation*}
	   \begin{split}
		\Big(M^{\pm}(\lambda)\Big)^{-1}=&\,\,\mp i\frac{4\pi}{\|V\|_{L^1}}S_{1}(S_{1}PS_{1})^{-1}S_{1}\eta^{-1}+\Big(D_{0}+\frac{16\pi^2}{\|V\|^{2}_{L^1}}S_{1}(S_{1}PS_{1})^{-1}S_{1}S_{1}vG_{2}vS_{1}(S_{1}PS_{1})^{-1}S_{1}\Big)\\
		&\,\,-\Big(D_{0}PS_{1}(S_{1}PS_{1})^{-1}S_{1}+S_{1}(S_{1}PS_{1})^{-1}S_{1}PD_{0}\Big)+O_{1}(\eta)	 	
	   \end{split}
	\end{equation*}
	in $B(0, 0)$ as $\eta\rightarrow0$.
	
	iii) If zero is purely an eigenvalue or zero is both resonance and eigenvalue of $H=\Delta^2-\Delta+V$ with $|V(x)|\lesssim(1+|x|)^{-\beta}$ for some $\beta>7$, then
	\begin{equation*}
		\Big(M^{\pm}(\lambda)\Big)^{-1}= \eta^{-2}S_{2}(S_{2}vG_{2}vS_{2})^{-1}S_{2}+\eta^{-1}A^{\pm}_{-1}+A^{\pm}_{0}+O_{1}(\eta)	 	
	\end{equation*}
in $B(0, 0)$ as $\eta\rightarrow0$. Here $A^{\pm}_{-1}$ and $A^{\pm}_{0}$ are Hilbert-Schmidt operators.
\end{theorem}

The following lemma  is the main tool to deduce the asymptotic  expansions of $\Big(M^{\pm}(\lambda)\Big)^{-1}$.
\begin{lemma}(\cite[Lemma 2.1]{JN})\label{inverseformula}
	Let $M$ be a closed operator on a Hilbert space $\mathscr{H}$ and $S$ be a projection. Suppose $M+S$ has a bounded inverse. Then $M$ has a bounded inverse if and only if
	\begin{equation*}
		M_{1}:=S-S(M+S)^{-1}S
	\end{equation*}
	has a bounded inverse in $S\mathscr{H}$, and in the case
	\begin{equation*}
		M^{-1}=(M+S)^{-1}+(M+S)^{-1}SM_{1}^{-1}S(M+S)^{-1}.
	\end{equation*}
\end{lemma}

\begin{proof}[\bf Proof of Theorem \ref{Mexpansions}]
Since $M^{-}(\lambda)=\overline{M^{+}}(\lambda)$, thus we only deal with $M^{+}(\lambda)$ below.

In the regular case, $T_{0}=U+vG_{0}V$ is invertible on $L^{2}(\mathbf{R}^3)$ and $$M^{+}(\lambda)=T_{0}+i\frac{\|V\|_{L^1}}{4\pi}P\eta+O_{1}(\eta^2).$$
Writing $\Big(M^{+}(\lambda)\Big)^{-1}$ into Neumann series, then
\begin{equation*} \Big(M^{+}(\lambda)\Big)^{-1}=\Big(1+i\frac{\|V\|_{L^1}}{4\pi}T_{0}^{-1}P\eta+O_{1}(\eta^2)\Big)^{-1}T_{0}^{-1}=T_{0}^{-1}-i\frac{\|V\|_{L^1}}{4\pi}T_{0}^{-1}PT_{0}^{-1}\eta+O_{1}(\eta^2).
\end{equation*}

If zero is a resonance of $H$, then $T_{0}+S_{1}$ is invertible since $S_{1}$ is the Riesz projection onto $\ker(T_{0})$ and $T_{0}$ is self-adjoint. Since
\begin{equation*}
	M^{+}(\lambda)=T_{0}+i\frac{\|V\|_{L^1}}{4\pi}P\eta+\eta^2 vG_{2}v+O_{1}(\eta^3),
\end{equation*}   	
applying Lemma \ref{inverseformula} to $M^{+}(\lambda)$ with projection $S_{1}$, then
\begin{equation}\label{Mresonace} \Big(M^{+}(\lambda)\Big)^{-1}=\Big(M^{+}(\lambda)+S_{1}\Big)^{-1}+\Big(M^{+}(\lambda)+S_{1}\Big)^{-1}S_{1}\Big(M_{1}^{+}(\lambda)\Big)^{-1}S_{1}\Big(M^{+}(\lambda)+S_{1}\Big)^{-1}
\end{equation}
where $M_{1}^{+}(\lambda)=S_{1}-S_{1}\Big(M^{+}(\lambda)+S_{1}\Big)^{-1}S_{1}$. Writing $\Big(M^{+}(\lambda)+S_{1}\Big)^{-1}$ into Neumann series, we have
\begin{equation*}
	\begin{split}
		\Big(M^{+}(\lambda)+S_{1}\Big)^{-1}=&\,\,\Big(1+i\frac{\|V\|_{L^1}}{4\pi}\eta D_{0}P+\eta^2 D_{0}vG_{2}v+O_{1}(\eta^3)\Big)^{-1}D_{0}\\
		=&\,\, D_{0}-i\frac{\|V\|_{L^1}}{4\pi}\eta D_{0}PD_{0}-\eta^2\Big(D_{0}vG_{2}vD_{0}+\frac{\|V\|_{L^{1}}^{2}}{16\pi^2}D_{0}PD_{0}PD_{0}\Big)+O_{1}(\eta^3)
	\end{split}
\end{equation*}
where $D_{0}=(T_{0}+S_{1})^{-1}$. Note that $\rm{rank}(P)=1$ since $P$ is a projection onto the span of $v$. Thus $PD_{0}P=\rho P$ with $\rho={\rm trace}(PD_{0}P)$. Using $S_{1}D_{0}=D_{0}S_{1}=S_{1}$, then
\begin{equation}
	M_{1}^{+}(\lambda)=i\frac{\|V\|_{L^1}}{4\pi}\eta S_{1}PS_{1}+\eta^2\Bigg(S_{1}vG_{2}vS_{1}+\frac{\|V\|_{L^{1}}^{2}}{16\pi^2}\rho S_{1}PS_{1}\Bigg)+O_{1}(\eta^3).
\end{equation}
In the resonance case, $T_{1}=S_{1}PS_{1}$ is invertible, thus
\begin{equation*}
	\begin{split}		\Big(M_{1}^{+}(\lambda)\Big)^{-1}=&\,\,\frac{4\pi}{i\|V\|_{L^1}}\eta^{-1}\Bigg(S_{1}PS_{1}+\frac{4\pi}{i\|V\|_{L^1}}\eta\Big(S_{1}vG_{2}vS_{1}+\frac{\|V\|_{L^{1}}^{2}}{16\pi^2}\rho S_{1}PS_{1}\Big)+O_{1}(\eta^2)\Bigg)^{-1}\\
		=&\,\, \frac{4\pi}{i\|V\|_{L^1}}\eta^{-1}T_{1}^{-1}+\frac{16\pi^2}{\|V\|_{L^1}^{2}}T_{1}^{-1}\Big(S_{1}vG_{2}vS_{1}+\frac{\|V\|_{L^{1}}^{2}}{16\pi^2}\rho S_{1}PS_{1}\Big)T_{1}^{-1}+O_{1}(\eta).
	\end{split}
\end{equation*}
Substituting the Neumann series of $\Big(M^{+}(\lambda)+S_{1}\Big)^{-1}$ and $\Big(M^{+}(\lambda)\Big)^{-1}$ into identity \eqref{Mresonace}, we obtain
\begin{equation*}
    \begin{split}  	\Big(M^{+}(\lambda)\Big)^{-1}=&-i\frac{4\pi}{\|V\|_{L^1}}S_{1}T_{1}^{-1}S_{1}\eta^{-1}+\Big(D_{0}+\frac{16\pi^2}{\|V\|^{2}_{L^1}}S_{1}T_{1}^{-1}S_{1}S_{1}vG_{2}vS_{1}T_{1}^{-1}S_{1}\Big)\\
    	&-\Big(D_{0}PS_{1}T_{1}^{-1}S_{1}+S_{1}T_{1}^{-1}S_{1}PD_{0}\Big)+O_{1}(\eta).	
    \end{split}
	\end{equation*}
	
	If  zero is an eigenvalue of $H$, then $S_{1}PS_{1}$ is not invertible but $S_{1}PS_{1}+S_{2}$ is invertible.  Denotes
	\begin{equation*}		\widetilde{M}_{1}^{+}(\lambda)=S_{1}PS_{1}+\frac{4\pi}{i\|V\|_{L^1}}\eta\Big(S_{1}vG_{2}vS_{1}+\frac{\|V\|_{L^{1}}^{2}}{16\pi^2}\rho S_{1}PS_{1}\Big)+O_{1}(\eta^2)
	\end{equation*}
	then $\Big(M_{1}^{+}(\lambda)\Big)^{-1}=\frac{4\pi}{i\|V\|_{L^1}}\eta^{-1}\Big(\widetilde{M}_{1}^{+}(\lambda)\Big)^{-1}$. Applying Lemma \ref{inverseformula} to $\widetilde{M}_{1}^{+}(\lambda)$ with projection $S_{2}$, then
	\begin{equation*}		\Big(\widetilde{M}_{1}^{+}(\lambda)\Big)^{-1}=\Big(\widetilde{M}_{1}^{+}(\lambda)+S_{2}\Big)^{-1}+\Big(\widetilde{M}_{1}^{+}(\lambda)+S_{2}\Big)^{-1}S_{2}\Big(M_{2}^{+}(\lambda)\Big)^{-1}S_{2}\Big(\widetilde{M}_{1}^{+}(\lambda)+S_{2}\Big)^{-1}
	\end{equation*}
	with $M_{2}^{+}(\lambda)=S_{2}-S_{2}\Big(\widetilde{M}_{1}^{+}(\lambda)+S_{2}\Big)^{-1}S_{2}$. Writing $\Big(\widetilde{M}_{1}^{+}(\lambda)+S_{2}\Big)^{-1}$ into Neumann series, then
	\begin{equation*}
	    \begin{split}
		\Big(\widetilde{M}_{1}^{+}(\lambda)+S_{2}\Big)^{-1}=&\,\,\Bigg(1+\frac{4\pi}{i\|V\|_{L^1}}\eta D_{1}\Big(S_{1}vG_{2}vS_{1}+\frac{\|V\|^{2}_{L^1}}{16\pi^2}\rho S_{1}PS_{1}\Big)+O_{1}(\eta^2)\Bigg)^{-1}D_{1}\\
		=&\,\, D_{1}-\frac{4\pi}{i\|V\|_{L^1}}\eta D_{1}\Big(S_{1}vG_{2}vS_{1}+\frac{\|V\|^{2}_{L^1}}{16\pi^2}\rho S_{1}PS_{1}\Big)D_{1}+O_{1}(\eta^2)
		\end{split}
	\end{equation*}
	where $D_{1}=(S_{1}PS_{1}+S_{2})^{-1}$. Using $S_{2}D_{1}=D_{1}S_{2}=S_{2}$, we get
	\begin{equation*}
		M_{2}^{+}(\lambda)=\frac{4\pi}{i\|V\|_{L^1}}\eta S_{2}\Big(S_{1}vG_{2}vS_{1}+\frac{\|V\|^{2}_{L^1}}{16\pi^2}\rho S_{1}PS_{1}\Big)S_{2}+O_{1}(\eta^2)=\frac{4\pi}{i\|V\|_{L^1}}\eta S_{2}vG_{2}vS_{2}+O_{1}(\eta^2).
	\end{equation*}
	Since  $S_{2}vG_{2}vS_{2}$ is invertible by Lemma \ref{stop},  then
	\begin{equation*}
		\Big(M_{2}^{+}(\lambda)\Big)^{-1}=i\frac{\|V\|_{L^1}}{4\pi}\eta^{-1}\Big(S_{2}vG_{2}vS_{2}\Big)^{-1}+O_{1}(1).
	\end{equation*}
	Substituting the expansions back step by step, we obtain
	$$\Big(M^{+}(\lambda)\Big)^{-1}=\eta^{-2}S_{2}(S_{2}vG_{2}vS_{2})^{-1}S_{2}+\eta^{-1}A^{+}_{-1}+A^{+}_{0}+O_{1}(\eta)$$
	with $A^{+}_{-1},\,\,A^{+}_{0}$ are Hilbert-Schmidt operators independent of $\eta$.
\end{proof}

\section{The perturbed evolution for low energy}\label{lowenergy}

In this section, our aim is to study the $L^{1}-L^{\infty}$ dispersive estimates of perturbed evolution $e^{-itH}$ for small energy. Here  small energy means the spectral variable $\lambda$ is near the threshold energy $\lambda=0$. The presence of zero resonance or zero eigenvalue affect the asymptotic behavior of the perturbed resolvent $R_{V}^{\pm}(\lambda)$ as $\lambda\rightarrow0$. The effect of the presence of zero energy resonance or zero eigenvalue is only felt in the small energy regime.
\begin{lemma}\label{freeresolventuniformlyestimates}
For $\lambda>0$ and $\eta=\big(\sqrt{1/4+\lambda}-1/2\big)^{1/2}$, then
\begin{equation}
	\sup_{x, y\in\mathbf{R}^3}\Big|R_{0}^{\pm}(\lambda; x, y)\Big|\lesssim \frac{\eta+\sqrt{1+\eta^2}}{1+2\eta^2}
\end{equation}
and
\begin{equation}
	\sup_{x, y\in\mathbf{R}^3}\Bigg|\frac{d}{d\eta}R_{0}^{\pm}(\lambda; x, y)\Bigg|\lesssim \frac{1}{1+2\eta^2}.
\end{equation}
\end{lemma}
\begin{proof}
	Recall that
	\begin{equation}
		R_{0}^{\pm}(\lambda; x, y)=\frac{1}{1+2\eta^2}\Bigg[\frac{e^{\pm i\eta|x-y|}}{4\pi|x-y|}-\frac{e^{-\sqrt{1+\eta^2}|x-y|}}{4\pi|x-y|}\Bigg]
	\end{equation}
	where $\eta=(\sqrt{1/4+\lambda}-1/2)^{1/2}$. By the mean value theorem, then
	\begin{equation*}
		\big|R_{0}^{\pm}(\lambda; x, y)\big|\lesssim \frac{1}{1+2\eta^2}\Bigg|\frac{e^{\pm i\eta|x-y|}-1}{4\pi|x-y|}+\frac{1-e^{-\sqrt{1+\eta^2}|x-y|}}{4\pi|x-y|}\Bigg|\lesssim \frac{\eta+\sqrt{1+\eta^2}}{1+2\eta^2}.
	\end{equation*}
	Furthermore, we have
	\begin{equation*}
	  \begin{split}
		 \Bigg|\frac{d}{d\eta}R_{0}^{\pm}(\lambda; x, y)\Bigg|=&\,\,\frac{4\eta}{(1+2\eta^2)^2}\Bigg|\frac{e^{\pm i\eta|x-y|}}{4\pi|x-y|}-\frac{e^{-\sqrt{1+\eta^2}|x-y|}}{4\pi|x-y|}\Bigg|\\
		 &\,\,+\frac{1}{1+2\eta^2}\Bigg|\frac{\pm ie^{\pm i\eta|x-y|}}{4\pi}+\frac{\eta(1+\eta^2)^{-1/2}e^{-\sqrt{1+\eta^2}|x-y|}}{4\pi}\Bigg|\\
		 \lesssim &\,\, \frac{\eta(\eta+\sqrt{1+\eta^2})}{(1+2\eta^2)^2}+\frac{1+\eta(1+\eta^2)^{-1/2}}{1+2\eta^2}
		 \lesssim \frac{1}{1+2\eta^2}.
	  \end{split}	
	\end{equation*}
\end{proof}

\begin{proposition}\label{regularlowenergy}
	For $H=\Delta^2-\Delta+V$ with $|V(x)|\lesssim(1+|x|)^{-\beta}$ for some $\beta>3$. Assume that $H$ has no positive embedded eigenvalue. If 0 is a regular point of $H$.

For $0<t\leq1$, then
	\begin{equation*}
		\sup_{x, y\in\mathbf{R}^3}\Bigg|\int_{0}^{t^{-1/4}}e^{-it(\eta^4+\eta^2)}\Big[R^{+}_{V}-R_{V}^{-}\Big](\eta^4+\eta^2; x, y)~(4\eta^3+2\eta)~d\eta\Bigg|\lesssim |t|^{-3/4}.
	\end{equation*}

For $t>1$, then
	\begin{equation*}
		\sup_{x, y\in\mathbf{R}^3}\Bigg|\int_{0}^{t^{-1/2}}e^{-it(\eta^4+\eta^2)}\Big[R^{+}_{V}-R_{V}^{-}\Big](\eta^4+\eta^2; x, y)~(4\eta^3+2\eta)~d\eta\Bigg|\lesssim |t|^{-3/2}.
	\end{equation*}
	
\end{proposition}
\begin{proof}
Recall that
\begin{equation*}
	R_{V}^{\pm}(\lambda)=R_{0}^{\pm}(\lambda)-R_{0}^{\pm}(\lambda)v\Big(M^{\pm}(\lambda)\Big)^{-1}vR_{0}^{\pm}(\lambda).
\end{equation*}	
It remains to prove the second term satisfies the above bounds by Proposition \ref{freedispersive}. In the regular case,  then $\big(M^{\pm}(\lambda)\big)^{-1}=T^{-1}_{0}+O^{\pm}(\eta)$ in $B(0, 0)$ by Theorem \ref{Mexpansions}. Furthermore, since $L^{\infty}(\mathbf{R}^3)\subset L^{2}_{-s}(\mathbf{R}^3)$ for $s>3/2$, then
\begin{equation}\label{22-1infty}
    \begin{split}
	  \Big\|R_{0}^{\pm}(\lambda)v\Big(M^{\pm}(\lambda)\Big)^{-1}vR_{0}^{\pm}(\lambda)\Big\|_{L^1\rightarrow L^{\infty}}\lesssim & ~\Big\|R_{0}^{\pm}(\lambda)\Big\|_{L^{2}_{s}\rightarrow L^{\infty}}\Big\|v\Big(M^{\pm}(\lambda)\Big)^{-1}v\Big\|_{L^{2}_{-s}\rightarrow L^{2}_{s}}\Big\|R_{0}^{\pm}(\lambda)\Big\|_{L^{1}\rightarrow L^{2}_{-s}}\\
	  \lesssim & ~\Big\|R_{0}^{\pm}(\lambda)\Big\|_{L^{1}\rightarrow L^{\infty}}\Big\|\Big(M^{\pm}(\lambda)\Big)^{-1}\Big\|_{L^{2}\rightarrow L^{2}}\Big\|R_{0}^{\pm}(\lambda)\Big\|_{L^{1}\rightarrow L^{\infty}}.
	 \end{split}
\end{equation}
By \eqref{differetfree}, \eqref{derivativefree} and Lemma \ref{freeresolventuniformlyestimates},  for $t>1$ then
\begin{equation*}
	\begin{split}
		&~\Bigg|\int_{0}^{t^{-1/2}}e^{-it(\eta^4+\eta^2)}\Big[R_{0}^{+}vT^{-1}_{0}vR_{0}^{+}-R_{0}^{-}vT^{-1}_{0}vR_{0}^{-}\Big](\eta^4+\eta^2)~(4\eta^3+2\eta)~d\eta\Bigg|\\
		\lesssim &~\Bigg|\int_{0}^{t^{-1/2}}e^{-it(\eta^4+\eta^2)}\Big[R_{0}^{+}vT^{-1}_{0}vR_{0}^{+}-R_{0}^{-}vT^{-1}_{0}vR_{0}^{+}\Big](\eta^4+\eta^2)d(\eta^4+\eta^2)\Bigg|\\
		&~+\Bigg|\int_{0}^{t^{-1/2}}e^{-it(\eta^4+\eta^2)}\Big[R_{0}^{-}vT^{-1}_{0}vR_{0}^{+}-R_{0}^{-}vT^{-1}_{0}vR_{0}^{-}\Big](\eta^4+\eta^2)d(\eta^4+\eta^2)\Bigg|\\
		\lesssim &~\frac{1}{|t|}~\Bigg|e^{-it(\eta^4+\eta^2)}\Big[\big(R_{0}^{+}-R_{0}^{-}\big)vT^{-1}_{0}vR_{0}^{+}\Big](\eta^4+\eta^2)\Big|_{0}^{t^{-1/2}}\Bigg|+\frac{1}{|t|}\int_{0}^{t^{-1/2}}\bigg|\frac{d}{d\eta}\Big[\big(R_{0}^{+}-R_{0}^{-}\big)vT^{-1}_{0}vR_{0}^{+}\Big](\eta^4+\eta^2)\bigg|d\eta\\
		\lesssim &~|t|^{-1}\Bigg|\eta\Big|_{0}^{t^{-1/2}}\Bigg|+\frac{1}{|t|}\int_{0}^{t^{-1/2}}\big|1+\eta\big|d\eta \lesssim |t|^{-3/2}.
	\end{split}
\end{equation*}
For the remaining term, we have
\begin{equation*}
	\begin{split}
		&~\Bigg|\int_{0}^{t^{-1/2}}e^{-it(\eta^4+\eta^2)}\Big[R_{0}^{\pm}(\eta^4+\eta^2)vO_{1}^{\pm}(\eta)vR_{0}^{\pm}(\eta^4+\eta^2)\Big]~(4\eta^3+2\eta)~d\eta\Bigg|\\
		\lesssim &~ \frac{1}{|t|}\Bigg|e^{-it(\eta^4+\eta^2)}\Big[R_{0}^{\pm}(\eta^4+\eta^2)vO_{1}^{\pm}(\eta)vR_{0}^{\pm}(\eta^4+\eta^2)\Big]\Big|_{0}^{t^{-1/2}}\Bigg|\\
		&~+\frac{1}{|t|}\int_{0}^{t^{-1/2}}\Big|\frac{d}{d\eta}\Big[R_{0}^{\pm}(\eta^4+\eta^2)vO_{1}^{\pm}(\eta)vR_{0}^{\pm}(\eta^4+\eta^2)\Big]\Big|d\eta\\
		\lesssim &~ |t|^{-3/2}+\frac{1}{|t|}\int_{0}^{t^{-1/2}}\big|1+\eta\big|d\eta\lesssim |t|^{-3/2}.
	\end{split}
\end{equation*}

For $0<t\leq 1$, by Lemma \ref{freeresolventuniformlyestimates} then
\begin{equation*}
	\begin{split}
		&~\Bigg|\int_{0}^{t^{-1/4}}e^{-it(\eta^4+\eta^2)}R_{0}^{\pm}(\eta^4+\eta^2)v\Big(M^{\pm}(\eta^4+\eta^2)\Big)^{-1}vR_{0}^{\pm}(\eta^4+\eta^2)~(4\eta^3+2\eta)~d\eta\Bigg|\\
		\lesssim &~ \int_{0}^{t^{-1/4}}\Big|(4\eta^3+2\eta)R_{0}^{\pm}(\eta^4+\eta^2)v\Big(M^{\pm}(\eta^4+\eta^2)\Big)^{-1}vR_{0}^{\pm}(\eta^4+\eta^2)\Big|d\eta\\
		\lesssim &~ \int_{0}^{t^{-1/4}}\Big|(4\eta^3+2\eta)\frac{\eta+\sqrt{1+\eta^2}}{1+2\eta^2}(1+\eta)\frac{\eta+\sqrt{1+\eta^2}}{1+2\eta^2}\Big|d\eta\\
		\lesssim &~ \int_{0}^{t^{-1/4}}(\eta+\eta^2)d\eta\lesssim |t|^{-3/4}.
	\end{split}
\end{equation*}
\end{proof}

\begin{proposition}\label{resonancelowenergy}
	For $H=\Delta^2-\Delta+V$ with $|V(x)|\lesssim(1+|x|)^{-\beta}$ for some $\beta>5$. Assume that $H$ has no positive embedded eigenvalue. If 0 is purely a resonance of $H$.
	
	For $0<t\leq1$, then
	\begin{equation*}
	\sup_{x, y\in\mathbf{R}^3}\Bigg|\int_{0}^{t^{-1/4}}e^{-it(\eta^4+\eta^2)}R^{\pm}_{V}(\eta^4+\eta^2; x, y)~(4\eta^3+2\eta)~d\eta\Bigg|\lesssim |t|^{-3/4}.
	\end{equation*}
	
	For $t>1$, then
	\begin{equation*}
		\sup_{x, y\in\mathbf{R}^3}\Bigg|\int_{0}^{t^{-1/2}}e^{-it(\eta^4+\eta^2)}\Big[R^{+}_{V}-R_{V}^{-}\Big](\eta^4+\eta^2; x, y)~(4\eta^3+2\eta)~d\eta-F(x, y)\Bigg|\lesssim |t|^{-3/2},
	\end{equation*}
	where $F$ is a time dependent finite rank operator satisfying $\|F\|_{L^{1}\rightarrow L^{\infty}}\lesssim |t|^{-1/2}$. Furthermore,
\begin{equation*}
  F(t;x,y)=\int_{0}^{t^{-1/2}}e^{-it(\eta^4+\eta^2)}\frac{4\pi (4\eta^2+2)}{i\|V\|_{L^1}}\bigg[R_{0}^{+}vS_{1}(S_{1}PS_{1})^{-1}S_{1}vR_{0}^{+}+R_{0}^{-}vS_{1}(S_{1}PS_{1})^{-1}S_{1}vR_{0}^{-}\bigg]d\eta.
\end{equation*}
\end{proposition}
\begin{proof}
	In the resonance case, by Theorem \ref{Mexpansions} in $B(0, 0)$:
	\begin{equation*}
	   \begin{split}
		\Big(M^{\pm}(\eta^4+\eta^2)\Big)^{-1}=&\,\,\mp i\frac{4\pi}{\|V\|_{L^1}}S_{1}(S_{1}PS_{1})^{-1}S_{1}\eta^{-1}-\Big(D_{0}PS_{1}(S_{1}PS_{1})^{-1}S_{1}+S_{1}(S_{1}PS_{1})^{-1}S_{1}PD_{0}\Big)\\
		&\,\,+\Big(D_{0}+\frac{16\pi^2}{\|V\|^{2}_{L^1}}S_{1}(S_{1}PS_{1})^{-1}S_{1}S_{1}vG_{2}vS_{1}(S_{1}PS_{1})^{-1}S_{1}\Big)+O_{1}(\eta)\\
		:=&\,\, M_{-1}^{\pm}\eta^{-1}+M_{0}+O_{1}(\eta).	 	
	   \end{split}
	\end{equation*}
	For $0<t\leq 1$, by Lemma \ref{freeresolventuniformlyestimates} and \eqref{22-1infty}, then
	\begin{equation*}
		\begin{split}
			&~\Bigg|\int_{0}^{t^{-1/4}}e^{-it(\eta^4+\eta^2)}R_{0}^{\pm}(\eta^4+\eta^2)v\Big(M^{\pm}(\eta^4+\eta^2)\Big)^{-1}vR_{0}^{\pm}(\eta^4+\eta^2)~(4\eta^3+2\eta)~d\eta\Bigg|\\
			\lesssim &~ \int_{0}^{t^{-1/4}}(4\eta^3+2\eta)\Big|R_{0}^{\pm}(\eta^4+\eta^2)v\Big(M^{\pm}(\eta^4+\eta^2)\Big)^{-1}vR_{0}^{\pm}(\eta^4+\eta^2)\Big|d\eta\\
			\lesssim &~ \int_{0}^{t^{-1/4}}(4\eta^3+2\eta)\frac{\eta+\sqrt{1+\eta^2}}{1+2\eta^2}(\eta^{-1}+1+\eta)\frac{\eta+\sqrt{1+\eta^2}}{1+2\eta^2}d\eta\\
			\lesssim &~\int_{0}^{t^{-1/4}}(1+\eta+\eta^2)d\eta\lesssim |t|^{-3/4}.
		\end{split}
	\end{equation*}
	For $t>1$,  we only need to deal with the term $R_{0}^{\pm}(\lambda)v\Big(M^{\pm}(\lambda)\Big)^{-1}vR_{0}^{\pm}(\lambda)$ by Proposition \ref{freedispersive} and the symmetric resolvent identity \eqref{symmetriresolventidentity}. Since
	\begin{equation*}
		\begin{split}
			&~R_{0}^{+}(\eta^4+\eta^2)v\Big(M^{+}(\eta^4+\eta^2)\Big)^{-1}vR_{0}^{+}(\eta^4+\eta^2)-R_{0}^{-}(\eta^4+\eta^2)v\Big(M^{-}(\eta^4+\eta^2)\Big)^{-1}vR_{0}^{-}(\eta^4+\eta^2)\\
			=&~R_{0}^{+}(\eta^4+\eta^2)v\Big(M_{-1}^{+}\eta^{-1}\Big)vR_{0}^{+}(\eta^4+\eta^2)-R_{0}^{-}(\eta^4+\eta^2)v\Big(M_{-1}^{-}\eta^{-1}\Big)vR_{0}^{-}(\eta^4+\eta^2)\\
			&+~R_{0}^{+}(\eta^4+\eta^2)vM_{0}vR_{0}^{+}(\eta^4+\eta^2)-R_{0}^{-}(\eta^4+\eta^2)vM_{0}vR_{0}^{-}(\eta^4+\eta^2)\\
			&+~R_{0}^{+}(\eta^4+\eta^2)vO_{1}^{+}(\eta)vR_{0}^{+}(\eta^4+\eta^2)-R_{0}^{-}(\eta^4+\eta^2)vO_{1}^{-}(\eta)vR_{0}^{-}(\eta^4+\eta^2)\\
			:=&~I(\eta; x, y)+II(\eta; x, y)+III(\eta; x, y).
		\end{split}
	\end{equation*}
	For $II(\eta; x, y)$ and $s>3/2$, we have
	\begin{equation*}
		\begin{split}
			~\bigg\|\Big[\big(R_{0}^{+}-R_{0}^{-}\big)vM_{0}vR_{0}^{+}\Big](\eta^4+\eta^2)\bigg\|_{L^{1}\rightarrow L^{\infty}}
			\lesssim &~ \Big\|R_{0}^{+}-R_{0}^{-}\Big\|_{L^{2}_{s}\rightarrow L^{\infty}}\Big\|vM_{0}v\Big\|_{L^{2}_{-s}\rightarrow L^{2}_{s}}\Big\|R_{0}^{+}\Big\|_{L^{1}\rightarrow L^{2}_{-s}}\\
			\lesssim &~ \Big\|R_{0}^{+}-R_{0}^{-}\Big\|_{L^{1}\rightarrow L^{\infty}}\Big\|M_{0}\Big\|_{L^{2}\rightarrow L^{2}}\Big\|R_{0}^{+}\Big\|_{L^{1}\rightarrow L^{\infty}}\\
			\lesssim &~\frac{\eta}{1+2\eta^2}\frac{\eta+\sqrt{1+\eta^2}}{1+2\eta^2}\lesssim \eta (1+2\eta^2)^{-3/2}.
		\end{split}
	\end{equation*}
	Similarly, by Lemma \ref{freeresolventuniformlyestimates} we have
	\begin{equation*}
		\Bigg\|\frac{d}{d\eta}\Big[\big(R_{0}^{+}-R_{0}^{-}\big)vM_{0}vR_{0}^{+}\Big](\eta^4+\eta^2)\Bigg\|_{L^{1}\rightarrow L^{\infty}}\lesssim \frac{\eta+\sqrt{1+\eta^2}}{(1+2\eta^2)^{2}}.
	\end{equation*}
	Thus we get
	\begin{equation*}
		\begin{split}
			&~\Bigg|\int_{0}^{t^{-1/2}}e^{-it(\eta^4+\eta^2)}II(\eta; x, y)~(4\eta^3+2\eta)~d\eta\Bigg|\\
			\lesssim &~ \Bigg|\int_{0}^{t^{-1/2}}e^{-it(\eta^4+\eta^2)}\Big[R_{0}^{+}-R_{0}^{-}\Big](\eta^4+\eta^2)vM_{0}vR_{0}^{+}(\eta^4+\eta^2)~d(\eta^4+\eta^2)\Bigg|\\
			&~\quad+\Bigg|\int_{0}^{t^{-1/2}}e^{-it(\eta^4+\eta^2)}R_{0}^{-}(\eta^4+\eta^2)vM_{0}v\Big[R_{0}^{+}-R_{0}^{-}\Big](\eta^4+\eta^2)~d(\eta^4+\eta^2)\Bigg|
			\end{split}
	\end{equation*}
	\begin{equation*}
	\begin{split}
			\lesssim &~ \frac{2}{|t|}\Bigg|e^{-it(\eta^4+\eta^2)}\Big[R_{0}^{+}-R_{0}^{-}\Big](\eta^4+\eta^2)vM_{0}vR_{0}^{+}(\eta^4+\eta^2)\Big|_{0}^{t^{-1/2}}\Bigg|\\
			&~\quad+\frac{2}{|t|}\int_{0}^{t^{-1/2}}\Bigg|\frac{d}{d\eta}\bigg(\Big[R_{0}^{+}-R_{0}^{-}\Big](\eta^4+\eta^2)vM_{0}vR_{0}^{+}(\eta^4+\eta^2)\bigg)\Bigg|d\eta\\
			\lesssim &~ |t|^{-3/2}+\frac{1}{|t|}\int_{0}^{t^{-1/2}}\frac{\eta+\sqrt{1+\eta^2}}{(1+2\eta^2)^2}d\eta \lesssim~ |t|^{-3/2}.
		\end{split}
	\end{equation*}
	
	For the third term $III(\eta; x, y)$, by the same argument as for $II(\eta; x, y)$,  we have
	\begin{equation*}
	  \begin{split}
		&~\Bigg|\int_{0}^{t^{-1/2}}e^{-it(\eta^4+\eta^2)}R_{0}^{\pm}(\eta^4+\eta^2)vO_{1}(\eta)vR_{0}^{\pm}(\eta^4+\eta^2)~(4\eta^3+2\eta)~d\eta\Bigg|\\
		\lesssim &~ \frac{1}{|t|}\Bigg|e^{-it(\eta^4+\eta^2)}R_{0}^{\pm}(\eta^4+\eta^2)vO_{1}(\eta)vR_{0}^{\pm}(\eta^4+\eta^2)\bigg|_{0}^{t^{-1/2}}\Bigg|\\
		&~+\frac{1}{|t|}\int_{0}^{t^{-1/2}}\Bigg|\frac{d}{d\eta}\Big[R_{0}^{\pm}(\eta^4+\eta^2)vO_{1}(\eta)vR_{0}^{\pm}(\eta^4+\eta^2)\Big]\Bigg|d\eta\\
		\lesssim &~|t|^{-3/2}+\frac{1}{|t|}\int_{0}^{t^{-1/2}}\frac{1}{1+2\eta^2}+\frac{\eta+\sqrt{1+\eta^2}}{(1+2\eta^2)^2}d\eta \lesssim |t|^{-3/2}.
	  \end{split}
	\end{equation*}
	
	Now, we show that the first term $I(\eta; x, y)$ only contributes $|t|^{-1/2}$. Since
	\begin{equation*}
		\begin{split}
			&~\Bigg|\int_{0}^{t^{-1/2}}e^{-it(\eta^4+\eta^2)}R_{0}^{\pm}(\eta^4+\eta^2)v\big(M_{-1}^{\pm}\eta^{-1}\big)vR_{0}^{\pm}(\eta^4+\eta^2)~(4\eta^3+2\eta)~d\eta\Bigg|\\
			\lesssim &~ \int_{0}^{t^{-1/2}}\Bigg|(4\eta^3+2\eta)R_{0}^{\pm}(\eta^4+\eta^2)v\big(M_{-1}^{\pm}\eta^{-1}\big)vR_{0}^{\pm}(\eta^4+\eta^2)\Bigg|~d\eta\\
			\lesssim &~ \int_{0}^{t^{-1/2}}\Bigg|(4\eta^3+2\eta)\frac{\eta+\sqrt{1+\eta^2}}{1+2\eta^2}\eta^{-1}\frac{\eta+\sqrt{1+\eta^2}}{1+2\eta^2}\Bigg|~d\eta\lesssim~ \int_{0}^{t^{-1/2}}1d\eta\lesssim |t|^{-1/2}.
		\end{split}
	\end{equation*}
	Let $F(t; x, y)=\int_{0}^{t^{-1/2}}e^{-it(\eta^4+\eta^2)}I(\eta; x, y)(4\eta^3+2\eta)~d\eta$. Then $\|F\|_{L^{1}\rightarrow L^{\infty}}\lesssim |t|^{-1/2}$. Furthermore, recall that the projection $S_{1}$ is of finite rank, then $M_{-1}^{\pm}=\mp i\frac{4\pi}{\|V\|_{L^1}}S_{1}\big(S_{1}PS_{1}\big)^{-1}S_{1}$ is of finite rank.
\end{proof}

\begin{proposition}\label{eigenvaluelowenergy}
	For $H=\Delta^2-\Delta+V$ with $|V(x)|\lesssim(1+|x|)^{-\beta}$ for some $\beta>7$. Assume that $H$ has no positive embedded eigenvalue. If 0 is purely an eigenvalue or 0 is both  resonance and eigenvalue of $H$.
	
For $0<t\leq1$, then
	\begin{equation}\label{eigensmallt}
		\sup_{x, y\in\mathbf{R}^3}\Bigg|\int_{0}^{t^{-1/4}}e^{-it(\eta^4+\eta^2)}\Big[R^{+}_{V}-R^{-}_{V}\Big](\eta^4+\eta^2; x, y)~(4\eta^3+2\eta)~d\eta\Bigg|\lesssim |t|^{-3/4}.
	\end{equation}

For $t>1$, then
	\begin{equation}\label{eigenlarget}
		\sup_{x, y\in\mathbf{R}^3}\Bigg|\int_{0}^{t^{-1/2}}e^{-it(\eta^4+\eta^2)}\Big[R^{+}_{V}-R_{V}^{-}\Big](\eta^4+\eta^2; x, y)~(4\eta^3+2\eta)~d\eta-G(x, y)\Bigg|\lesssim |t|^{-3/2}
	\end{equation}
	where $G$ is a time dependent finite rank operator satisfying $\|G\|_{L^{1}\rightarrow L^{\infty}}\lesssim |t|^{-1/2}$. Furthermore,
\begin{equation*}
G(t;x,y)=F(t;x,y)+\int_{0}^{t^{-1/2}}e^{-it(\eta^4+\eta^2)}(4\eta+\frac{2}{\eta})\bigg[R_{0}^{+}vS_{2}(S_{2}vG_{2}vS_{2})^{-1}S_{2}vR_{0}^{+}-R_{0}^{-}vS_{2}(S_{2}vG_{2}vS_{2})^{-1}S_{2}vR_{0}^{-}\bigg]d\eta.
\end{equation*}
\end{proposition}
\begin{proof}
Recall that $0$ is both resonance and eigenvalue of $H$ if $S_{2}=S_{1}\neq0$. $0$ is purely an eigenvalue of $H$ if $S_{2}< S_{1}$ strictly, see Remark \ref{propertyofS}. Here we deal with the case $S_{2}<S_{1}$. The following argument also holds for $S_{2}=S_{1}$.
	If zero is an eigenvalue of $H$,  by Theorem \ref{Mexpansions}, then
	\begin{equation*}
		\Big(M^{\pm}(\eta^4+\eta^2)\Big)^{-1}=\eta^{-2}A_{-2}+\eta^{-1}A_{-1}^{\pm}+A_{0}^{\pm}+O_{1}(\eta).
	\end{equation*}
	where $A_{-2}=S_{2}\Big(S_{2}vG_{2}vS_{2}\Big)^{-1}S_{2}$. Applying the similar argument as in Theorem \ref{resonancelowenergy}, we know for $0<t\leq 1$, the term $\eta^{-1}A_{-1}^{\pm}$, $A_{0}^{\pm}$ and the remaining term $O_{1}(\eta)$ satisfy the expected estimates \eqref{eigensmallt}. For $t>1$, $A_{0}^{\pm}$ and $O_{1}(\eta)$ contributes $|t|^{-3/2}$ in  estimate \eqref{eigenlarget}. The term $\eta^{-1}A_{-1}^{\pm}$ contributes $|t|^{-3/2}$ in estimate \eqref{eigenlarget}.
	
	Next, we show the first term $\eta^{-2}A_{-2}$ satisfies the estimates we needed.  By \eqref{22-1infty} and Lemma \ref{freeresolventuniformlyestimates}, then
	\begin{equation*}
		\begin{split}
			&~\Bigg|\int_{0}^{t^{-1/2}}e^{-it(\eta^4+\eta^2)}\Big[R^{+}_{0}(\eta^4+\eta^2)v\big(\eta^{-2}A_{-2}\big)vR^{+}_{0}(\eta^4+\eta^2)-R^{-}_{0}(\eta^4+\eta^2)v\big(\eta^{-2}A_{-2}\big)vR^{-}_{0}(\eta^4+\eta^2)\Big]~d(\eta^4+\eta^2)\Bigg|\\
			\lesssim &~\Bigg|\int_{0}^{t^{-1/2}}e^{-it(\eta^4+\eta^2)}\Big[R^{+}_{0}-R^{-}_{0}\Big](\eta^4+\eta^2)v\big(\eta^{-2}A_{-2}\big)vR^{+}_{0}(\eta^4+\eta^2)~d(\eta^4+\eta^2)\Bigg|\\
			&~\quad+ \Bigg|\int_{0}^{t^{-1/2}}e^{-it(\eta^4+\eta^2)}R^{-}_{0}(\eta^4+\eta^2)v\big(\eta^{-2}A_{-2}\big)v\Big[R^{+}_{0}-R^{-}_{0}\Big](\eta^4+\eta^2)~d(\eta^4+\eta^2)\Bigg|\\
			\lesssim &~ \int_{0}^{t^{-1/2}}\Bigg|(4\eta^3+2\eta)\Big[R^{+}_{0}-R^{-}_{0}\Big](\eta^4+\eta^2)v\big(\eta^{-2}A_{-2}\big)vR^{+}_{0}(\eta^4+\eta^2)\Bigg|~d\eta	\\
			\lesssim &~ 	\int_{0}^{t^{-1/2}}\Bigg|(4\eta^3+2\eta)\frac{\eta}{1+2\eta^2}\eta^{-2}\frac{\eta+\sqrt{1+\eta^2}}{1+2\eta^2}\Bigg|~d\eta\lesssim  \int_{0}^{t^{-1/2}}1~d\eta\lesssim |t|^{-1/2}.
		\end{split}
	\end{equation*}
	Similarly, we get the expected estimates for $0<t\leq1$ with $|t|^{-1/4}$. Let
	\begin{equation*}
		\begin{split}
			G(t; x, y)=F(t; x, y)+\int_{0}^{t^{-1/2}}&e^{-it(\eta^4+\eta^2)}\bigg[R^{+}_{0}(\eta^4+\eta^2)v\big(\eta^{-2}A_{-2}\big)vR^{+}_{0}(\eta^4+\eta^2)\\
			&-R^{-}_{0}(\eta^4+\eta^2)v\big(\eta^{-2}A_{-2}\big)vR^{-}_{0}(\eta^4+\eta^2)\bigg]~(4\eta^3+2\eta)~d\eta.
		\end{split}
	\end{equation*}
	Thus $G(t; x, y)$ satisfies $\big\|G\big\|_{L^{1}\rightarrow L^{\infty}}\lesssim |t|^{-1/2}$ for $t>1$. Furthermore, $S_{2}$ is  finite rank by Remark \ref{propertyofS}. Hence $G$ is  finite rank.
\end{proof}

\section{The perturbed evolution for high energy}\label{highenergy}
In this part, we aim to show the proof of  dispersive bound for high energy portion of the perturbed evolution.  This section include two subsections. In the first subsection, under the assumption that $H$ has no positive embedded eigenvalue, we prove decay estimates for $R^{\pm}_{V}(\lambda)$ with $\lambda\rightarrow\infty$ in $B(s, -s)$ using the limiting absorption principle. In the second subsection, we show the high energy bounds of Theorem \ref{timedecay}.

\subsection{High energy decay estimates of $R_{V}(z)$}
Denote by $\mathbf{C}^{+}=\{\rm{Im}\,z>0\}$ the open upper half complex plane, and by $\mathbf{C}^{-}=\{\rm{Im}\,z<0\}$ the open lower half complex plane. Define $\Xi$ be the disjoint union of $\overline{\mathbf{C}^{+}}$ and $\overline{\mathbf{C}^{-}}$ with the identified points $z\leq0$. Denote  $V_0:=\min\{V(x), x\in\mathbf{R}^{d}\}$. Since for any $\lambda\in \mathbf{R}\setminus [V_0, +\infty)$, $H-\lambda=\Delta^{2}-\Delta+V-\lambda>0$, then for the resolvent set $\rho(H)$ of $H$, we have $\mathbf{C}\setminus [V_0, +\infty)\subseteq\rho(H)$.
\begin{theorem}\label{continuity}
Let $k=0,1,2,3,\cdots$, and $|V(x)|\lesssim(1+|x|)^{-\beta}$ with $\beta>k+1$. Assuming that $H$ has no positive embedded eigenvalue. For any $s>k+1/2$,  then $R^{(k)}_{V}(z)\in B(s, -s)$ is continuous for $z\in\Xi\setminus[V_0, 0]$. Further, the bound
\begin{equation}\label{extendhighenergy}
  \big\|R^{(k)}_{V}(z)\big\|_{B(s, -s)}
  = O\big(|z|^{-(3+3k)/4}\big)
\end{equation}
holds as $z\rightarrow \infty$ in $\Xi\setminus[V_0, 0]$.
\end{theorem}
\begin{proof}
The theorem follows by Proposition \ref{Perturbed Resolvent} and Lemma \ref{limitabsorp large} below.
\end{proof}

Using the high energy decay estimates for free resolvent $R_{0}(z)$ , we first prove estimate \eqref{extendhighenergy} for $z\in\Xi\setminus[V_0, \infty)$ by perturbation arguments. Then, we extend $z$ to the positive real axis applying the limiting absorption principle. For the free resolvent we have:

\begin{proposition}\label{freehighenergydecay}
For $z\in\mathbf{C}\setminus[0,+\infty)$, $k=0,1,2,\cdots$ and  any $s, s'>k+\frac{1}{2}$,  then
\begin{equation}\label{23}
  \big\|R_{0}^{(k)}(z)\big\|_{B(s, -s')}\lesssim |z|^{-(3+3k)/4},\,\,\,\, |z|\rightarrow \infty.
\end{equation}
\end{proposition}
\begin{proof}
For free resolvent of Laplacian $R(-\Delta; z):=(-\Delta-z)^{-1}$, it is known that
\begin{equation}\label{22}
  \big\|R^{(k)}(-\Delta; z)\big\|_{B(s, -s')}\lesssim |z|^{-(1+k)/2},\quad |z|\rightarrow\infty
\end{equation}
holds, see e.g. \cite[Theorem 16.1]{KK}. For $k=0$, by \eqref{freeresolventidentity} and \eqref{22}, then
\begin{align*}
  \big\|R_{0}(z)\big\|_{B(s, -s')}
  & = \Big|\frac{1}{2\sqrt{1/4+z}}\Big|\bigg\|\Big(-\Delta+\frac{1}{2}-\sqrt{1/4+z}\Big)^{-1}-\Big(-\Delta+\frac{1}{2}+\sqrt{1/4+z}\Big)^{-1}\bigg\|_{B(s, -s')}\\
  & \lesssim \Big|\frac{1}{2\sqrt{1/4+z}}\Big|\Bigg(\bigg\|\Big(-\Delta+\frac{1}{2}-\sqrt{1/4+z}\Big)^{-1}\bigg\|_{B(s, -s')}+\bigg\|\Big(-\Delta+\frac{1}{2}+\sqrt{1/4+z}\Big)^{-1}\bigg\|_{B(s, -s')}\Bigg)\\
  & \lesssim \Big|\frac{1}{2\sqrt{1/4+z}}\Big|\Bigg(\Big(\sqrt{1/4+z}-1/2\Big)^{-1/2}+\Big(\sqrt{1/4+z}+1/2\Big)^{-1/2}\Bigg)\\
  & \lesssim |z|^{-3/4},\,\,\, |z|\rightarrow\infty.
\end{align*}
Now we check decay estimate \eqref{23} for $k\neq0$. For any $s, s'>1/2+k$, then
\begin{align*}
  \big\|R^{(k)}_{0}(z)\big\|_{B(s, -s')}
  & \lesssim \Big|\frac{d^{k_1}}{dz^{k_1}}\frac{1}{2\sqrt{1/4+z}}\Big|\Bigg\|\frac{d^{k_2}}{dz^{k_2}}\Big(-\Delta+\frac{1}{2}-\sqrt{1/4+z}\Big)^{-1}-\frac{d^{k_2}}{dz^{k_2}}\Big(-\Delta+\frac{1}{2}+\sqrt{1/4+z}\Big)^{-1}\Bigg\|_{B(s, -s')}\\
  & \lesssim \big|z\big|^{-1/2-k_1}\Bigg(\Bigg\|\frac{d^{k_2}}{dz^{k_2}}\Big(-\Delta+\frac{1}{2}-\sqrt{1/4+z}\Big)^{-1}\Bigg\|_{B(s, -s')}+\Bigg\|\frac{d^{k_2}}{dz^{k_2}}\Big(-\Delta+\frac{1}{2}+\sqrt{1/4+z}\Big)^{-1}\Bigg\|_{B(s, -s')}\Bigg)\\
  & \lesssim \big|z\big|^{-1/2-k_1}\Bigg(\Bigg\|\Big(\frac{1}{\zeta}\frac{d}{d\zeta}\Big)^{k_2}\Big(-\Delta+\frac{1}{2}-\zeta\Big)^{-1}\Bigg\|_{B(s, -s')}+\Bigg\|\Big(\frac{1}{\zeta}\frac{d}{d\zeta}\Big)^{k_2}\Big(-\Delta+\frac{1}{2}+\zeta\Big)^{-1}\Bigg\|_{B(s, -s')}\Bigg)\\
  & \lesssim |z|^{-(3+3k)/4},\,\,\, |z|\rightarrow\infty.
\end{align*}
where $\zeta=\sqrt{1/4+z}$ and $k=k_1+k_2$ ($k_1, k_2\geq0$).
\end{proof}

Now, we prove the high energy decay estimate of perturbed resolvent $R^{(k)}_{V}(z)$.
\begin{lemma}\label{Fredholm}
Let $V(x)$ satisfies that $|V(x)|\lesssim(1+|x|)^{-\beta}$ with some $\beta>2$. Then operators $vR_{0}(z)v\in B(0, 0)$ are compact for $z\in \mathbf{C}\setminus[0, +\infty)$. Moreover, $M(z)=U+vR_{0}(z)v$ are invertible in $L^{2}(\mathbf{R}^{3})$ for $z\in\mathbf{C}\setminus[V_0, +\infty)$.
\end{lemma}
\begin{proof}
By the free resolvent identity \eqref{freeresolventidentity}, we only need to show  $v\big(-\Delta+\frac{1}{2}\pm\sqrt{1/4+z}~\big)^{-1}v$ are compact  in $L^{2}(\mathbf{R}^3)$. Recall that the kernel $k(\xi; x, y)$ of $\big(-\Delta-\xi\big)^{-1}$:
\begin{equation*}
	k(\xi; x, y)=e^{-i\sqrt{\xi}|x-y|}/(4\pi|x-y|),\,\,\, \rm{Im}\sqrt{\xi}>0.
\end{equation*}
Thus, under the assumption on $V$, one can check that the Hilbert-Schmidt norm
$$\int_{\mathbf{R}^3}\int_{\mathbf{R}^3}|v(x)k(\xi; x, y)v(y)|^2 dxdy<\infty.$$
Next, we show that $M(z)=U+vR_0(z)v$ is invertible by  Fredholm's alternative theorem.  We claim that $(H-z)\psi=0$ only has trivial solution in $L^{2}(\mathbf{R}^3)$.  In fact, for $z\in \mathbf{C}\setminus\mathbf{R}$, if $\psi\neq0$, then
		\begin{equation*}
			{\rm Im}\big((H-z)\psi, \psi\big)=-{\rm Im}\, z(\psi, \psi)\neq0.
		\end{equation*}
		Since $\big((H-z)\psi, \psi\big)=(H\psi, \psi)-z(\psi, \psi)$ and $(H\psi, \psi)\in\mathbf{R}$, hence $(H-z)\psi\neq0$ if $\psi\neq0$.  For $z\in\mathbf{C}\setminus\mathbf{R}$ and ${\rm Re} (z)< V_{0}$, then
		\begin{equation*}
		{\rm Re}\big((H-z)\psi, \psi\big)\geq {\rm Re} (V_{0}-z)(\psi, \psi)\neq0
		\end{equation*}
		provided $\psi\neq0$. Thus $(H-z)\psi\neq0$ if $\psi\neq0$.
		
		Let $\psi=Uv\phi$. Then  $M(z)\phi=0$ if and only if $(H-z)\psi=0$. Thus $M(z)\phi=0$ only has zero solution. Hence, Fredholm's alternative theorem tells that $M(z)=U+vR_0(z)v$ is invertible.
\end{proof}
\begin{proposition}\label{Perturbed Resolvent}
Let $k=0,1,2,\cdots$, and  $|V(x)|\lesssim(1+|x|)^{-\beta}$ with $\beta>2k+2$. For large $z\in\mathbf{C}\setminus[V_0, +\infty)$ and any $s, s'>k+1/2$, then
\begin{equation}\label{29}
  \big\|R^{(k)}_{V}(z)\big\|_{B(s, -s')}
  \leq C(\sigma, k)\big(|z|^{-(3+3k)/4}\big),\quad |z|\rightarrow\infty.
\end{equation}
\end{proposition}
\begin{proof}
For $k=0$, by  identity \eqref{symmetriresolventidentity} and Proposition \ref{freehighenergydecay}, the above bound \eqref{29} holds by the uniformly boundedness of $M(z)^{-1}$ for large $z\in\mathbf{C}\setminus[V_0, +\infty)$ in $L^{2}$.  It is equivalent to prove that for large $z\in\mathbf{C}\setminus[0,+\infty),$
		\begin{equation*}\label{eq-f}
			\big\|f\big\|_{L^2(\mathbf{R}^3)}\le C\big\|\big(U+vR_0(z)v\big)f\big\|_{L^2(\mathbf{R}^3)}.
		\end{equation*}
		In fact, by the triangle inequality we have
		\begin{equation*}\label{eq-triangle}
			\big|\|f\|_{L^2}-\|vR_0(z)vf\|_{L^2}\big|\le\big\|\big(U+vR_0(z)v\big)f\big\|_{L^2}\le \|f\|_{L^2}+\big\|vR_0(z)vf\big\|_{L^2}.
		\end{equation*}
		By the decay estimate \eqref{23}, for $|z|$ large enough, then
		\begin{equation*}
			\big\|vR_0(z)vf\big\|_{L^2}\le C(s, a)|z|^{-\frac{2m-1}{2m}}\|f\|_{L^2}\leq \frac{1}{4}\|f\|_{L^2}.
		\end{equation*}
		
		For $k\geq1$,  differentiating \eqref{symmetriresolventidentity} $k$-times in $z$, then
		\begin{equation}
			R^{(k)}_{V}(z)=R^{(k)}_{0}(z)-\sum\limits_{k_1+k_2+k_3=k}R^{(k_1)}_{0}(z)v\frac{d^{k_2}}{dz^{k_2}}\Big[M(z)^{-1}\Big]vR^{(k_3)}_{0}(z).
		\end{equation}
		Note that the derivative term $\frac{d^{k_2}}{dz^{k_2}}\big[M(z)^{-1}\big]$ is a linear combination of terms such as $\big[M(z)^{-1}\big]^{j}M^{(\ell)}(z)$ with $0\leq j, \ell< k_{2}$. By the representation of $M(z)$, we know $M^{(\ell)}(z)=vR^{(\ell)}_{0}(z)v$	for $\ell\geq 1$. Since $v(x)(1+|x|)^{\ell+1/2}\in L^{\infty}$ under the assumption of $V(x)$, thus \eqref{29} holds by mathematical induction.
\end{proof}

Now, we show the limiting absorption principle for $R_{V}(z)$.  For the resolvent of free Schr\"odinger operator $R(-\Delta; \zeta)$,  we have:
\begin{lemma}
( \cite[Theorem 8.1]{JK} ) Let $k=0,1,2,\cdots$. If $s>k+1/2$, then $R^{(k)}(-\Delta; \zeta)\in B(s, -s)$ is continuous in $\zeta \in\Xi\setminus\{0\}$. Further, the boundary value
\begin{equation*}
  R^{(k)}(-\Delta; \lambda\pm i0)=\lim_{\epsilon\downarrow0}R^{(k)}(-\Delta; \lambda\pm i\epsilon)\in B(s, -s)
\end{equation*}
exists for any $\lambda\in(0, +\infty)$. The decay estimate \eqref{22} can extend $\zeta\in\mathbf{C}\setminus[0, +\infty)$ to $\zeta\in\Xi\setminus\{0\}$.
\end{lemma}
By the resolvent identity \eqref{freeresolventidentity}, for the free resolvent $R_{0}(z)$,  we have:
\begin{corollary}\label{freehighabsorp}
Let $k=0,1,2,\cdots$. For $s>k+1/2$, then $R^{(k)}_{0}(z)\in B(s, -s)$
is continuous in $z\in \Xi\setminus \{0\}$. Further, the boundary value
\begin{equation*}
  R^{(k)}_{0}(\lambda\pm i0)=\lim_{\epsilon\downarrow0}R^{(k)}_{0}(\lambda\pm i\epsilon)\in B(s, -s)
\end{equation*}
exists for any $\lambda\in(0,\, +\infty)$, and  the bound
\begin{equation}\label{freehigh}
   \big\|R^{(k)}_{0}(z)\big\|_{B(s,-s)}=O\big(|z|^{-(3+3k)/4}\big)
\end{equation}
holds as $z\rightarrow \infty$ in $\Xi\setminus \{0\}$.
\end{corollary}

Next, under the spectral assumption  that $H$ has no positive embedded eigenvalues, we  prove that the boundary value $R_{V}^{\pm}(\lambda)$ exists on $\lambda\in(0,\, +\infty)$. Recall that $V_0$ is the minimal value of potential function $V$. If $V_0\geq0$, then the segment $[V_0,\, 0]=\{0\}$.

\begin{lemma}\label{inverseabsorp}
	 Let $|V(x)|\lesssim (1+|x|)^{-\beta}$ with some $\beta>2$. For $\lambda\geq0$,   then $vR_0^{\pm}(\lambda)v\in B(0, 0)$  are compact.
\end{lemma} 	
\begin{proof}
	Recall the kernel of $R_{0}^{\pm}(\lambda)$:
	\begin{equation*}
		R_{0}^{\pm}(\lambda; x, y)=\frac{1}{1+2\eta^2}\Bigg(\frac{e^{\pm i\eta|x-y|}}{4\pi|x-y|}-\frac{e^{-\sqrt{1+\eta^2}|x-y|}}{4\pi|x-y|}\Bigg),
	\end{equation*}
	where $\eta=\Big(\sqrt{1/4+\lambda}-1/2\Big)^{1/2}$. Thus for $\lambda\geq0$, then $\Big|R_{0}^{\pm}(\lambda; x, y)\Big|\leq \frac{1}{2\pi|x-y|}$.
	Since the Hilbert-Schmidt norm of $v(x)|x-y|^{-1}v(y)$ is finite under the assumption of $V$, thus $vR_0^{\pm}(\lambda)v$ are compact.
\end{proof}

\begin{lemma}\label{limitabsorp large}
		For $H=\Delta^2-\Delta+V$, assume that $H$ has no positive embedded eigenvalue.
		\par (1) Let $|V(x)|\lesssim (1+|x|)^{-\beta}$ with some $\beta>2$. For
		$s, s'>1/2$, then $R_V(z)\in B(s, -s')$ is continuous for $z\in{\bf  \Xi}\setminus(\Sigma\cup\{0\})$. Furthermore, the boundary value
		\[R_V^{\pm}(\lambda)=\lim_{\epsilon\downarrow0}R_V(\lambda\pm i\epsilon)\in B(s, -s')\]
		exists for $\lambda\in \sigma_c(H)\setminus (\Sigma\cup\{0\})$.
		\par (2) Let $|V(x)|\lesssim (1+|x|)^{-\beta}$ with some $\beta>2$. Assume that $0$ is a regular point of $H$. For $s, s'>1$, then the function $R_V(z)\in B(s, -s')$ defined on $z\in{\bf  \Xi}\setminus\Sigma$ is continuous at $z=0$.
\end{lemma}
\begin{proof}
		The conclusions follow from Lemma \ref{inverseabsorp} and  symmetric resolvent identity \eqref{symmetriresolventidentity} provided
		\begin{equation*}
			\big[U+vR_{0}(\lambda\pm i\epsilon)v\big]^{-1}\rightarrow \big[U+vR_{0}(\lambda\pm i0)v\big]^{-1},\,\, \epsilon\downarrow 0.
		\end{equation*}
		The convergence holds if and only if both limit operators $U+vR_{0}(\lambda\pm i0)v: L^2(\mathbf{R}^3)\rightarrow L^2(\mathbf{R}^3) $ are invertible. According to Lemma \ref{inverseabsorp} and Fredholm's alternative theorem, it is enough to show that $\big[U+vR_{0}(\lambda\pm i0)v\big]\phi=0$ only admits zero solution in $L^{2}(\mathbf{R}^3)$. Note that $\big[U+vR_{0}(\lambda\pm i0)v\big]\phi=0$ equals $(H-\lambda)\psi=0$ with $\psi=Uv\phi$. Thus $\phi=0$ under the assumptions that zero is a regular point of $H$ and $H$ has no positive eigenvalue.
\end{proof}

\subsection{Proof of dispersive estimate for high energy}
To deal with the high energy part, we use the resolvent identity:
\begin{equation}\label{limitbornseries}
R_{V}^{\pm}(\lambda)=R_{0}^{\pm}(\lambda)-R_{0}^{\pm}(\lambda)VR_{0}^{\pm}(\lambda)+R_{0}^{\pm}(\lambda)VR_{V}^{\pm}(\lambda)VR_{0}^{\pm}(\lambda).
\end{equation}
From Proposition \ref{freedispersive}, we know that the first summand in \eqref{limitbornseries} satisfies the expected estimates. Therefore, it suffices to establish the expected bounds  for the last two summands in \eqref{limitbornseries}. Recall that $\lambda=\eta^4+\eta^2$. Denote
\begin{align*}
&\mathcal{R}_{0}^{\pm}(\eta; x, y):=\Big[R_{0}^{\pm}(\lambda)VR_{0}^{\pm}(\lambda)\Big](x, y)=R_{0}^{\pm}VR_{0}^{\pm}(\eta; x, y);\\
&\mathcal{R}_{V}^{\pm}(\eta; x, y):=\Big[R_{0}^{\pm}(\lambda)VR_{V}^{\pm}(\lambda)VR_{0}^{\pm}(\lambda)\Big](x, y)=R_{0}^{\pm}VR_{V}^{\pm}VR_{0}^{\pm}(\eta; x, y).
\end{align*}

For the term $\mathcal{R}_{0}^{\pm}(\eta; x, y)=R_{0}^{\pm}VR_{0}^{\pm}(\eta; x, y)$, we have:
\begin{proposition}\label{R0VR0}
	Let $|V(x)|\lesssim (1+|x|)^{-\beta}$ for some $\beta>3$.

For $t>1$, then
	\begin{equation*}
		\sup_{x, y\in\mathbf{R}^3}\Bigg|\int_{t^{-1/2}}^{\infty}e^{-it(\eta^4+\eta^2)}\Big[\mathcal{R}_{0}^{+}-\mathcal{R}_{0}^{-}\Big](\eta; x, y)~(4\eta^3+2\eta)~d\eta\Bigg|\lesssim |t|^{-3/2}.
	\end{equation*}
	
For $0<t\leq 1$, then
	\begin{equation*}
		\sup_{x, y\in\mathbf{R}^3}\Bigg|\int_{t^{-1/4}}^{\infty}e^{-it(\eta^4+\eta^2)}\mathcal{R}_{0}^{\pm}(\eta; x, y)~(4\eta^3+2\eta)~d\eta\Bigg|\lesssim |t|^{-3/4}.
	\end{equation*}
\end{proposition}
\begin{proof}
By Lemma \ref{freeresolventuniformlyestimates}, then
	\begin{align*}
		\Big|\mathcal{R}_{0}^{\pm}(\eta; x, y)\Big|\lesssim \|V\|_{L^{1}(\mathbf{R}^3)}\Big(\sup_{x, y\in\mathbf{R}^3}\big|R_{0}^{\pm}(\lambda; x, y)\big|\Big)^2\lesssim \frac{1}{1+2\eta^2}
	\end{align*}
	and
	\begin{align*}
		&\Bigg|\frac{d}{d\eta}\mathcal{R}_{0}^{\pm}(\eta; x, y)\Bigg|\lesssim \|V\|_{L^{1}(\mathbf{R}^3)}\bigg(\sup_{x, y\in\mathbf{R}^3}\Big|\frac{d}{d\eta}R_{0}^{\pm}(\lambda; x, y)\Big|\bigg)\Big(\sup_{x, y\in\mathbf{R}^3}\big|R_{0}^{\pm}(\lambda; x, y)\big|\Big)\lesssim \frac{1+\eta}{(1+2\eta^2)^2}.
	\end{align*}
 For $0<t\leq1$, integrating by parts, then
 \begin{equation*}
   \begin{split}
 	&~\Bigg|\int_{t^{-1/4}}^{\infty}e^{-it(\eta^4+\eta^2)}\mathcal{R}_{0}^{\pm}(\eta; x, y)~(4\eta^3+2\eta)~d\eta\Bigg|\\
 	\lesssim &~ \frac{1}{|t|}\Bigg|e^{-it(\eta^4+\eta^2)}\mathcal{R}_{0}^{\pm}(\eta; x, y)\Big|_{t^{-1/4}}^{\infty}\Bigg|+\frac{1}{|t|}\int_{t^{-1/4}}^{\infty}\bigg|\frac{d}{d\eta}\mathcal{R}_{0}^{\pm}(\eta; x, y)\bigg|d\eta\\
 	\lesssim &~ \frac{1}{|t|(1+|t|^{-1/2})}+\frac{1}{|t|}\frac{1}{(1+|t|^{-1/4})^2}\lesssim |t|^{-3/4},\,\, 0<t\leq1. 	
 	\end{split}
 \end{equation*}
  For $t>1$,  we have:
 \begin{equation*}
 \begin{split}
 &~\Bigg|\int_{t^{-1/2}}^{\infty}e^{-it(\eta^4+\eta^2)}\Big[\mathcal{R}_{0}^{+}-\mathcal{R}_{0}^{-}\Big](\eta; x, y)~(4\eta^3+2\eta)~d\eta\Bigg|\\
 \lesssim&~\Bigg|\int_{t^{-1/2}}^{\infty}e^{-it(\eta^4+\eta^2)}\Big[\big(R_{0}^{+}-R_{0}^{-}\big)VR_{0}^{+}\Big](\eta; x, y)~d(\eta^4+\eta^2)\Bigg|\\
 &~+\Bigg|\int_{t^{-1/2}}^{\infty}e^{-it(\eta^4+\eta^2)}\Big[R_{0}^{-}V\big(R_{0}^{+}-R_{0}^{-}\big)\Big](\eta; x, y)~d(\eta^4+\eta^2)\Bigg|\\
 := &~ I(t)+II(t).
 \end{split}
 \end{equation*}
 For the first term $I(t)$, integrating by parts, then
 \begin{equation*}
 \begin{split}
 I(t)\lesssim&~\frac{1}{|t|}\Bigg|e^{-it(\eta^4+\eta^2)}\Big[\big(R_{0}^{+}-R_{0}^{-}\big)VR_{0}^{+}\Big](\eta; x, y)\Big|_{t^{-1/2}}^{\infty}\Bigg|+\frac{1}{|t|}\bigg|\int_{t^{-1/2}}^{\infty}e^{-it(\eta^4+\eta^2)}\frac{d}{d\eta}\Big[\big(R_{0}^{+}-R_{0}^{-}\big)VR_{0}^{+}\Big](\eta; x, y)d\eta\bigg|\\
 \lesssim&~\frac{1}{|t|}\Bigg|e^{-it(\eta^4+\eta^2)}\frac{1}{(1+2\eta^2)^2}\int_{\mathbf{R}^3}\Bigg(\frac{e^{i\eta|x-x_{1}|}-e^{-i\eta|x-x_{1}|}}{4\pi|x-x_{1}|}\Bigg)V(x_{1})\Bigg(\frac{e^{i\eta|y-x_1|}-e^{-\sqrt{1+\eta^2}|y-x_{1}|}}{4\pi|y-x_{1}|}\Bigg)dx_{1}\Big|_{t^{-1/2}}^{\infty}\Bigg|\\
 &~+\frac{1}{|t|}\bigg|\int_{t^{-1/2}}^{\infty}e^{-it(\eta^4+\eta^2)}\frac{d}{d\eta}\Big[\big(R_{0}^{+}-R_{0}^{-}\big)VR_{0}^{+}\Big](\eta; x, y)d\eta\bigg|\\
 \lesssim&~|t|^{-3/2}\Bigg|\frac{1}{(1+2\eta^2)^2}\int_{\mathbf{R}^3}\Bigg(\frac{e^{i\eta|x-x_{1}|}-e^{-i\eta|x-x_{1}|}}{4\pi\eta|x-x_{1}|}\Bigg)V(x_{1})\Bigg(\frac{e^{i\eta|y-x_1|}-e^{-\sqrt{1+\eta^2}|y-x_{1}|}}{4\pi|y-x_{1}|}\Bigg)dx_{1}\Big|_{\eta=t^{-1/2}}\Bigg|\\
&~+\frac{1}{|t|}\bigg|\int_{t^{-1/2}}^{\infty}e^{-it(\eta^4+\eta^2)}\frac{d}{d\eta}\Big[\big(R_{0}^{+}-R_{0}^{-}\big)VR_{0}^{+}\Big](\eta; x, y)d\eta\bigg|\\
\lesssim&~|t|^{-3/2}\|V\|_{L^{1}(\mathbf{R}^3)}+\frac{1}{|t|}\bigg|\int_{t^{-1/2}}^{\infty}e^{-it(\eta^4+\eta^2)}\frac{d}{d\eta}\Big[\big(R_{0}^{+}-R_{0}^{-}\big)VR_{0}^{+}\Big](\eta; x, y)d\eta\bigg|.
 \end{split}
 \end{equation*}
Now, we aim to show that
\begin{equation*}
\sup_{x, y\in\mathbf{R}^3}\bigg|\int_{t^{-1/2}}^{\infty}e^{-it(\eta^4+\eta^2)}\frac{d}{d\eta}\Big[\big(R_{0}^{+}-R_{0}^{-}\big)VR_{0}^{+}\Big](\eta; x, y)d\eta\bigg|\lesssim |t|^{-1/2}, \quad t>1.
\end{equation*}
Since
\begin{equation*}
\begin{split}
&\frac{d}{d\eta}\Big[\big(R_{0}^{+}-R_{0}^{-}\big)VR_{0}^{+}\Big](\eta; x, y)\\=&\frac{d}{d\eta}\Bigg[\frac{1}{(1+2\eta^2)^2}\int_{\mathbf{R}^3}\Bigg(\frac{e^{i\eta|x-x_{1}|}-e^{-i\eta|x-x_{1}|}}{4\pi|x-x_{1}|}\Bigg)V(x_{1})\Bigg(\frac{e^{i\eta|y-x_1|}-e^{-\sqrt{1+\eta^2}|y-x_{1}|}}{4\pi|y-x_{1}|}\Bigg)dx_{1}\Bigg]\\
=&~\frac{-8\eta}{(1+2\eta^2)^3}\int_{\mathbf{R}^3}\Bigg(\frac{e^{i\eta|x-x_{1}|}-e^{-i\eta|x-x_{1}|}}{4\pi|x-x_{1}|}\Bigg)V(x_{1})\Bigg(\frac{e^{i\eta|y-x_1|}-e^{-\sqrt{1+\eta^2}|y-x_{1}|}}{4\pi|y-x_{1}|}\Bigg)dx_{1}\\
&~+\frac{i}{(1+2\eta^2)^2}\int_{\mathbf{R}^3}\Bigg(\frac{e^{i\eta|x-x_{1}|}+e^{-i\eta|x-x_{1}|}}{4\pi}\Bigg)V(x_{1})\Bigg(\frac{e^{i\eta|y-x_1|}-e^{-\sqrt{1+\eta^2}|y-x_{1}|}}{4\pi|y-x_{1}|}\Bigg)dx_{1}\\
&~+\frac{1}{(1+2\eta^2)^2}\int_{\mathbf{R}^3}\Bigg(\frac{e^{i\eta|x-x_{1}|}-e^{-i\eta|x-x_{1}|}}{4\pi|x-x_{1}|}\Bigg)V(x_{1})\Bigg(\frac{ie^{i\eta|y-x_1|}+\eta(1+\eta^2)^{-1/2}e^{-\sqrt{1+\eta^2}|y-x_{1}|}}{4\pi}\Bigg)dx_{1}.
\end{split}
\end{equation*}
Furthermore, note that
\begin{equation*}
\sup_{x, y\in\mathbf{R}^3}\Bigg|\frac{e^{i\eta|x-x_{1}|}-e^{-i\eta|x-x_{1}|}}{4\pi|x-x_{1}|}\Bigg|\lesssim \eta; \quad \sup_{x, y\in\mathbf{R}^3}\Bigg|\frac{e^{i\eta|y-x_1|}-e^{-\sqrt{1+\eta^2}|y-x_{1}|}}{4\pi|y-x_{1}|}\Bigg| \lesssim \eta+\sqrt{1+\eta^2};
\end{equation*}
and
\begin{equation*}
\sup_{x, y\in\mathbf{R}^3}\Bigg|\frac{d}{d\eta}\frac{e^{i\eta|x-x_{1}|}-e^{-i\eta|x-x_{1}|}}{4\pi|x-x_{1}|}\Bigg|\lesssim 1; \quad \sup_{x, y\in\mathbf{R}^3}\Bigg|\frac{d}{d\eta}\frac{e^{i\eta|y-x_1|}-e^{-\sqrt{1+\eta^2}|y-x_{1}|}}{4\pi|y-x_{1}|}\Bigg| \lesssim 1+\eta(1+\eta^2)^{-1/2}.
\end{equation*}
By van der Corput's lemma, then
\begin{equation*}
\sup_{x, y\in\mathbf{R}^3}\bigg|\int_{t^{-1/2}}^{\infty}e^{-it(\eta^4+\eta^2)}\int_{\mathbf{R}^3}\frac{-8\eta}{(1+2\eta^2)^3}\Bigg(\frac{e^{i\eta|x-x_{1}|}-e^{-i\eta|x-x_{1}|}}{4\pi|x-x_{1}|}\Bigg)V(x_{1})\Bigg(\frac{e^{i\eta|y-x_1|}-e^{-\sqrt{1+\eta^2}|y-x_{1}|}}{4\pi|y-x_{1}|}\Bigg)dx_{1}d\eta\bigg|\lesssim |t|^{-1/2}.
\end{equation*}

For the remaining terms, similarly,  we have
\begin{equation*}
\begin{split}
&\bigg|\int_{t^{-1/2}}^{\infty}e^{-it(\eta^4+\eta^2)}\frac{i}{(1+2\eta^2)^2}\int_{\mathbf{R}^3}\Bigg(\frac{e^{i\eta|x-x_{1}|}+e^{-i\eta|x-x_{1}|}}{4\pi}\Bigg)V(x_{1})\Bigg(\frac{e^{i\eta|y-x_1|}-e^{-\sqrt{1+\eta^2}|y-x_{1}|}}{4\pi|y-x_{1}|}\Bigg)dx_{1}d\eta\bigg|\\
=~&\bigg|\int_{\mathbf{R}^3}\int_{t^{-1/2}}^{\infty}\Big(e^{-it(\eta^4+\eta^2-\frac{\eta|x-x_{1}|}{t})}+e^{-it(\eta^4+\eta^2+\frac{\eta|x-x_{1}|}{t})}\Big)\Bigg(\frac{e^{i\eta|y-x_1|}-e^{-\sqrt{1+\eta^2}|y-x_{1}|}}{16\pi^2(1+2\eta^2)^2|y-x_{1}|}\Bigg)d\eta V(x_{1})dx_{1}\bigg|
\lesssim~ |t|^{-1/2}.
\end{split}
\end{equation*}
Applying van der Corput's lemma again, we also obtain
\begin{equation*}
\begin{split}
&\bigg|\int_{t^{-1/2}}^{\infty}e^{-it(\eta^4+\eta^2)}\frac{1}{(1+2\eta^2)^2}\int_{\mathbf{R}^3}\Bigg(\frac{e^{i\eta|x-x_{1}|}-e^{-i\eta|x-x_{1}|}}{4\pi|x-x_{1}|}\Bigg)V(x_{1})\Bigg(\frac{ie^{i\eta|y-x_1|}+\eta(1+\eta^2)^{-1/2}e^{-\sqrt{1+\eta^2}|y-x_{1}|}}{4\pi}\Bigg)dx_{1}d\eta\bigg|\\
\lesssim~&\bigg|\int_{\mathbf{R}^3}\int_{t^{-1/2}}^{\infty}e^{-it(\eta^4+\eta^2-\eta|y-x_{1}|/t)}\Bigg(\frac{e^{i\eta|x-x_1|}-e^{-i\eta|x-x_{1}|}}{16\pi^2(1+2\eta^2)^2|x-x_{1}|}\Bigg)d\eta V(x_{1})dx_{1}\bigg|\\
&~+\bigg|\int_{t^{-1/2}}^{\infty}e^{-it(\eta^4+\eta^2)}\int_{\mathbf{R}^3}\frac{\eta(1+\eta^2)^{-1/2}}{16\pi^2 (1+2\eta^2)^2}\Bigg(\frac{e^{i\eta|x-x_1|}-e^{i\eta|x-x_{1}|}}{|x-x_{1}|}\Bigg)V(x_{1})e^{-\sqrt{1+\eta^2}|y-x_1|}dx_{1}d\eta\bigg|\\
\lesssim&~ |t|^{-1/2}.
\end{split}
\end{equation*}
Hence, we get $I(t)\lesssim |t|^{-3/2}$ for $t>1$. 	Applying  the similar arguments to $II(t)$ as for $I(t)$, we can obtain  $II(t)\lesssim |t|^{-3/2}$ for $t>1$.
 \end{proof}

For the last term $\mathcal{R}_{V}^{\pm}(\eta; x, y)=R_{0}^{\pm}VR_{V}^{\pm}VR_{0}^{\pm}(\eta; x, y)$, we will use the weighted-$L^2$ estimates of $R_{V}^{\pm}(\lambda)$. From the low energy asymptotic expansions Theorem \ref{Mexpansions} and the high energy decay estimates Theorem \ref{continuity}, $R_{V}^{\pm}(\eta^4+\eta^2)$ satisfy different estimates for $\eta$  small and large. Thus for the term $\mathcal{R}_{V}^{\pm}(\eta; x, y)$ with large time $t>1$, we split $t^{-1/2}<\eta<\infty$ into middle part $t^{-1/2}<\eta<t$ and large part $\eta>t$.

For small time $0<t\leq1$ and large time $t>1$ with high energy $\eta>t$, we have:
\begin{proposition}\label{highdecay}
	For $H=\Delta^2-\Delta+V$ with $|V(x)|\lesssim (1+|x|)^{-\beta}$ for some $\beta>3$. Assume that $H$ has no positive embedded eigenvalue.

For $t>1$, then
	\begin{equation*}\label{perturbedhighlarget}
		\sup_{x, y\in\mathbf{R}^3}\Bigg|\int_{t}^{\infty}e^{-it(\eta^4+\eta^2)}\Big[\mathcal{R}_{V}^{+}(\lambda; x, y)-\mathcal{R}_{V}^{-}(\lambda; x, y)\Big]~(4\eta^3+2\eta)~d\eta\Bigg|\lesssim |t|^{-3/2}.
	\end{equation*}
	
For $0<t\leq1$, then
	\begin{equation*}\label{perturbedhighsmallt}
		\sup_{x, y\in\mathbf{R}^3}\Bigg|\int_{t^{-1/4}}^{\infty}e^{-it(\eta^4+\eta^2)}\Big[\mathcal{R}_{V}^{+}(\lambda; x, y)-\mathcal{R}_{V}^{-}(\lambda; x, y)\Big]~(4\eta^3+2\eta)~d\eta\Bigg|\lesssim |t|^{-3/4}.
	\end{equation*}
\end{proposition}

\begin{proof}
 Note that $L^{\infty}(\mathbf{R}^3)\subset L^{2}_{-s}(\mathbf{R}^3)$ for any $s>3/2$. From Lemma \ref{freeresolventuniformlyestimates} and  Theorem \ref{continuity}, then
	\begin{equation*}
	    \begin{split}
		  \Big\|\mathcal{R}_{V}^{\pm}(\eta)\Big\|_{L^1\rightarrow L^{\infty}}
		  \leq\,\,\, & \big\|R_{0}^{\pm}(\lambda)\big\|_{L^{2}_{s}\rightarrow L^{\infty}}\big\|VR_{V}^{\pm}(\lambda)V\big\|_{L^{2}_{-s}\rightarrow L^{2}_{s}}\big\|R_{0}^{\pm}(\lambda)\big\|_{L^{1}\rightarrow L^{2}_{-s}}\\
		  \leq\,\,\, & \big\|R_{0}^{\pm}(\lambda)\big\|_{L^1\rightarrow L^{\infty}}\|V\|_{L^{2}_{-s}\rightarrow L^{2}_{s}}\big\|R_{V}^{\pm}(\lambda)\big\|_{L^{2}_{s}\rightarrow L^{2}_{-s}}\|V\|_{L^{2}_{-s}\rightarrow L^{2}_{s}}\big\|R_{0}^{\pm}(\lambda)\big\|_{L^{1}\rightarrow L^{\infty}}\\
		  \lesssim\,\,\, & \Bigg(\frac{\eta+\sqrt{1+\eta^2}}{1+2\eta^2}\Bigg)^2\Big(\eta^4+\eta^2\Big)^{-3/4}\\
		  \lesssim\,\,\, & \frac{1}{1+2\eta^2}\frac{\eta^{-3/2}}{(1+\eta^{2})^{3/4}}\lesssim \frac{1}{\eta^{3/2}(1+\eta^2)^{7/4}}\lesssim \eta^{-1}.
		\end{split}
	\end{equation*}
	Similarly, for $1<\eta$, we have
	\begin{equation*}
		\begin{split}
			\Big\|\frac{d}{d\eta}\mathcal{R}_{V}^{\pm}(\eta)\Big\|_{L^1\rightarrow L^{\infty}}
			\lesssim \,\,\, &\Big\|\frac{d}{d\eta}R_{0}^{\pm}(\lambda)\Big\|_{L^2_{s}\rightarrow L^{\infty}}\|V\|_{L^{2}_{-s}\rightarrow L^{2}_{s}}\big\|R_{V}^{\pm}(\lambda)\big\|_{L^{2}_{s}\rightarrow L^{2}_{-s}}\|V\|_{L^{2}_{-s}\rightarrow L^{2}_{s}}\big\|R_{0}^{\pm}(\lambda)\big\|_{L^{1}\rightarrow L^2_{-s}}\\
			&+\big\|R_{0}^{\pm}(\lambda)\big\|_{L^2_{s}\rightarrow L^{\infty}}\|V\|_{L^{2}_{-s}\rightarrow L^{2}_{s}}\big\|\frac{d\lambda}{d\eta}\frac{d}{d\lambda}R_{V}^{\pm}(\lambda)\big\|_{L^{2}_{s}\rightarrow L^{2}_{-s}}\|V\|_{L^{2}_{-s}\rightarrow L^{2}_{s}}\big\|R_{0}^{\pm}(\lambda)\big\|_{L^{1}\rightarrow L^{2}_{-s}}\\
			\lesssim \,\,\, & \Bigg(\frac{\eta(\eta+\sqrt{1+\eta^2})}{(1+2\eta^2)^2}+\frac{1+\eta(1+\eta^2)^{-1/2}}{1+2\eta^2}\Bigg)\Big(\eta^4+\eta^2\Big)^{-3/4}\frac{\eta+\sqrt{1+\eta^2}}{1+2\eta^2}\\
			&+\Bigg(\frac{\eta+\sqrt{1+\eta^2}}{1+2\eta^2}\Bigg)^2\Big(\eta^4+\eta^2\Big)^{-3/2}\Big(4\eta^3+2\eta\Big)\\
			\lesssim\,\,\, & \frac{1}{\eta^{3/2}(1+\eta^2)^{9/4}}+\frac{1}{\eta^2(1+\eta^2)^{3/2}}\lesssim \eta^{-2}.
		\end{split}
	\end{equation*}
	Integrating by parts,  for $0<t\leq1$ then
	\begin{equation*}
	\begin{split}
 		~\bigg|\int_{t^{-1/4}}^{\infty}e^{-it(\eta^4+\eta^2)}\mathcal{R}_{V}^{\pm}(\eta; x, y)d(\eta^4+\eta^2)\bigg|
 		\lesssim& ~ \frac{1}{|t|}\Bigg|e^{-it(\eta^4+\eta^2)}\mathcal{R}_{V}^{\pm}(\eta; x, y)\Big|_{t^{-1/4}}^{\infty}\Bigg|
 		+\frac{1}{|t|}\int_{t^{-1/4}}^{\infty}\bigg|\frac{d}{d\eta}\mathcal{R}_{V}^{\pm}(\eta; x, y)\bigg|d\eta\\
 		\lesssim &~ \frac{1}{|t|}\Big(|t|^{-1/4}\Big)^{-1}
 		+\frac{1}{|t|}|t|^{1/4}\lesssim|t|^{-3/4}.
 	\end{split}
 \end{equation*}
 For $t>1$, we have
 	\begin{equation*}
 \begin{split}
 \bigg|\int_{t}^{\infty}e^{-it(\eta^4+\eta^2)}\mathcal{R}_{V}^{\pm}(\eta; x, y)d(\eta^4+\eta^2)\bigg|
 \lesssim &~ \int_{t}^{\infty}\bigg|(4\eta^3+2\eta)\mathcal{R}_{V}^{\pm}(\eta; x, y)\bigg|d\eta\\
 \lesssim &~ \int_{t}^{\infty}\bigg|\frac{4\eta^3+2\eta}{\eta^{3/2}(1+\eta^2)^{7/4}}\bigg|d\eta\lesssim|t|^{-3/2}.
 \end{split}
 \end{equation*}
\end{proof}


For large time $t>1$ with middle energy $t^{-1/2}<\eta<t$, we have:
\begin{proposition}\label{middledecay}
	For $H=\Delta^2-\Delta+V$ with $|V(x)|\lesssim (1+|x|)^{-\beta}$ for some $\beta>0$. Assume that $H$ has no positive embedded eigenvalue.
	
	(1)~ If $0$ is a regular point of $H$, let $\beta>3$, then
	\begin{equation*}\label{middleregular}
	\sup_{x, y\in\mathbf{R}^3}\Bigg|\int_{t^{-1/2}}^{t}e^{-it(\eta^4+\eta^2)}\Big[\mathcal{R}_{V}^{+}(\lambda; x, y)-\mathcal{R}_{V}^{-}(\lambda; x, y)\Big]~(4\eta^3+2\eta)~d\eta\Bigg|\lesssim |t|^{-3/2},\,\, t>1.
	\end{equation*}
	
 	(2)~ If $0$ is a resonance of $H$, let $\beta>5$, then
 \begin{equation*}\label{middleregular}
 \sup_{x, y\in\mathbf{R}^3}\Bigg|\int_{t^{-1/2}}^{t}e^{-it(\eta^4+\eta^2)}\Big[\mathcal{R}_{V}^{+}(\lambda; x, y)-\mathcal{R}_{V}^{-}(\lambda; x, y)\Big]~(4\eta^3+2\eta)~d\eta\Bigg|\lesssim |t|^{-1/2},\,\, t>1.
 \end{equation*}

 	(3)~  If $0$ is a resonance and~/~or eigenvalue of $H$, let $\beta>7$, then
 \begin{equation*}\label{middleregular}
 \sup_{x, y\in\mathbf{R}^3}\Bigg|\int_{t^{-1/2}}^{t}e^{-it(\eta^4+\eta^2)}\Big[\mathcal{R}_{V}^{+}(\lambda; x, y)-\mathcal{R}_{V}^{-}(\lambda; x, y)\Big]~(4\eta^3+2\eta)~d\eta\Bigg|\lesssim |t|^{-1/2},\,\, t>1.
 \end{equation*}
\end{proposition}

Since  $\|R_{V}^{\pm}(\lambda)\|_{B(s, -s)}$ is continuous for $\lambda>0$ with suitable $s>0$, thus $\|R_{V}^{\pm}(\eta^4+\eta^2)\|_{B(s, -s)}$ for $t^{-1/2}<\eta<t$  is controlled by $\|R_{V}^{\pm}(\eta^4+\eta^2)\|_{B(s, -s)}$ for $\eta$ near zero or $\eta$ approaches $\infty$.  However, the presence of zero resonance or eigenvalue affect the asymptotic propertity of $R_{V}^{\pm}(\lambda)$, see Theorem \ref{Mexpansions}. Thus for the middle energy $t^{-1/2}<\eta<t~(t>1)$, $R_{V}^{\pm}(\eta^4+\eta^2) $ satisfies the following estimates:
\begin{lemma}\label{middleRV}
For $H=\Delta^2-\Delta+V$ with $|V(x)|\lesssim (1+|x|)^{-\beta}$ for some $\beta>0$. Assume that $H$ has no positive embedded eigenvalue. Then
\begin{equation*}
\Big\|R_{V}^{\pm}(\eta^4+\eta^2)\Big\|_{B(s, -s)}\lesssim \begin{cases}
1,\,\, & \,\, 0~ \text{is regular};\\ \eta^{-1}\,\, & \,\, 0~ \text{is a resonance};\\  \eta^{-2},\,\, & \,\, 0~ \text{is a resonance and/or eigenvalue};
\end{cases}
\end{equation*}
and
\begin{equation*}
\Big\|\big(R_{V}^{+}-R_{V}^{-}\big)(\eta^4+\eta^2)\Big\|_{B(s, -s)}\lesssim \begin{cases}
\eta,\,\, & \,\, 0~ \text{is regular};\\ \eta^{-1}\,\, & \,\, 0~ \text{is a resonance};\\  \eta^{-1},\,\, & \,\, 0~ \text{is a resonance and/or eigenvalue};
\end{cases}
\end{equation*}
and
\begin{equation*}
\Big\|\frac{d}{d\eta}R_{V}^{\pm}(\eta^4+\eta^2)\Big\|_{B(s, -s)}\lesssim \begin{cases}
1,\,\, & \,\, 0~ \text{is regular};\\ \eta^{-2}\,\, & \,\, 0~ \text{is a resonance};\\  \eta^{-3},\,\, & \,\, 0~ \text{is a resonance and/or eigenvalue};
\end{cases}
\end{equation*}
and
\begin{equation*}
\Big\|\frac{d}{d\eta}\big(R_{V}^{+}-R_{V}^{-}\big)(\eta^4+\eta^2)\Big\|_{B(s, -s)}\lesssim \begin{cases}
1,\,\, & \,\, 0~ \text{is regular};\\ \eta^{-2}\,\, & \,\, 0~ \text{is a resonance};\\  \eta^{-2},\,\, & \,\, 0~ \text{is a resonance and/or eigenvalue}.
\end{cases}
\end{equation*}
Here we choose $\beta>3,\,\, s>3/2$ if 0 is a regular point, $\beta>5,\,\, s>5/2$ if 0 is a resonance and $\beta>7,\,\, s>7/2$ if 0 is a resonance and~/~or eigenvalue.
\end{lemma}
\begin{proof}
	Recall that
	\begin{equation*}
	R_{V}^{\pm}(\eta^4+\eta^2)=R_{0}^{\pm}(\eta^4+\eta^2)-R_{0}^{\pm}(\eta^4+\eta^2)v\Big(M^{\pm}(\eta^4+\eta^2)\Big)^{-1}vR_{0}^{\pm}(\eta^4+\eta^2)
	\end{equation*}
	where $v(x)=|V(x)|^{1/2}$. Thus
\begin{equation*}
  R_{V}^{+}-R_{V}^{-}=\Big(R_{0}^{+}-R_{0}^{-}\Big)-\Big(R_{0}^{+}v(M^{+})^{-1}vR_{0}^{+}-R_{0}^{-}v(M^{-})^{-1}vR_{0}^{-}\Big).
\end{equation*}
By the high energy decay estimates (Theorem \ref{continuity}), for large $\eta$, we get 
$$\big\|R_{V}^{\pm}(\eta^4+\eta^2)\big\|_{B(s,-s)}\lesssim |\eta|^{-3},\quad \Big\|\frac{d}{d\eta}R_{V}^{\pm}(\eta^4+\eta^2)\Big\|_{B(s,-s)}\lesssim |\eta|^{-3}.$$
Substituting the asymptotic expansion of $R_{0}^{\pm}$ and expansions of $\big(M^{\pm}(\eta^4+\eta^2)\big)^{-1}$ (Theorem \ref{Mexpansions}) into the above identities of $R_{V}^{\pm}$ and $R_{V}^{+}-R_{V}^{-}$, we get the desired estimates of $R_{V}^{\pm}$ and $R_{V}^{+}-R_{V}^{-}$.

For the estimates of the first derivative, since
	\begin{equation*}
	\begin{split}
	\frac{d}{d\eta}R_{V}^{\pm}(\eta^4+\eta^2)=&~\Big[\frac{d}{d\eta}R_{0}^{\pm}(\eta^4+\eta^2)\Big]-\Big[\frac{d}{d\eta}R_{0}^{\pm}(\eta^4+\eta^2)\Big]v\Big(M^{\pm}(\eta^4+\eta^2)\Big)^{-1}vR_{0}^{\pm}(\eta^4+\eta^2)\\
	&~-R_{0}^{\pm}(\eta^4+\eta^2)v\Big[\frac{d}{d\eta}\Big(M^{\pm}(\eta^4+\eta^2)\Big)^{-1}\Big]vR_{0}^{\pm}(\eta^4+\eta^2)\\
	&~-R_{0}^{\pm}(\eta^4+\eta^2)v\Big(M^{\pm}(\eta^4+\eta^2)\Big)^{-1}v\Big[\frac{d}{d\eta}R_{0}^{\pm}(\eta^4+\eta^2)\Big].
	\end{split}
	\end{equation*}
	Further, 
	\begin{equation*}
	\frac{d}{d\eta}\Big(M^{\pm}(\eta^4+\eta^2)\Big)^{-1}=\frac{d}{d\eta}\Big(U+vR_{0}^{\pm}(\eta^4+\eta^2)v\Big)^{-1}=\Big(M^{\pm}(\eta^4+\eta^2)\Big)^{-1}v\Big[\frac{d}{d\eta}R_{0}^{\pm}(\eta^4+\eta^2)\Big]v\Big(M^{\pm}(\eta^4+\eta^2)\Big)^{-1}.
	\end{equation*}
	Substituting asymptotic expansions of $\Big(M^{\pm}(\eta^4+\eta^2)\Big)^{-1}$ and $\frac{d}{d\eta}R_{0}^{\pm}(\eta^4+\eta^2)$ (see Lemma \ref{freeresolventexpansion} and Theorem \ref{Mexpansions})  into the preceding identity, then  we obtain the asymptotic expansion of $\frac{d}{d\eta}\Big(M^{\pm}(\eta^4+\eta^2)\Big)^{-1}$ as $\eta\rightarrow 0$.  Espescially, when 0 is a resonance and~/~or eigenvalue, then
	\begin{equation*}
	\begin{split}
	\frac{d}{d\eta}\Big(M^{\pm}(\eta^4+\eta^2)\Big)^{-1}=&~\Big(\eta^{-2}S_{2}(S_2vG_2vS_2)^{-1}S_{2}+\eta^{-1}A_{-1}^{\pm}+A_{0}^{\pm}+O(\eta)\Big)v\Big(\pm iG_{1}+2\eta G_2+O(\eta^2)\Big)v\\
	&~\cdot \Big(\eta^{-2}S_{2}(S_2vG_2vS_2)^{-1}S_{2}+\eta^{-1}A_{-1}^{\pm}+A_{0}^{\pm}+O(\eta)\Big).
	\end{split}
	\end{equation*}
	Recall that $PS_{2}=S_{2}P=0$. Thus the coefficient of the term of $\eta^{-4}$ is $$S_{2}(S_2vG_2vS_2)^{-1}S_{2}vG_1vS_{2}(S_2vG_2vS_2)^{-1}S_{2}=\frac{\|V\|_{L^1}}{4\pi}S_{2}(S_2vG_2vS_2)^{-1}S_{2}PS_{2}(S_2vG_2vS_2)^{-1}S_{2}=0.$$
Moreover, the coefficient of the term of $\eta^{-3}$ is 
$$2S_{2}(S_2vG_2vS_2)^{-1}S_{2}G_{2}S_{2}(S_2vG_2vS_2)^{-1}S_{2}$$
which is selfadjoint.	
By Lemma \ref{freeresolventuniformlyestimates}, then
	$$\sup_{\eta\geq0}\big\|R_{0}^{\pm}(\eta^4+\eta^2)\big\|_{B(s, -s)}\lesssim 1, \quad \sup_{\eta\geq0}\big\|\big(R_{0}^{+}-R_{0}^{-}\big)(\eta^4+\eta^2)\big\|_{B(s, -s)}\lesssim \eta;$$
and
$$\sup_{\eta\geq0}\Big\|\frac{d}{d\eta}R_{0}^{\pm}(\eta^4+\eta^2)\Big\|_{B(s, -s)}\lesssim 1, \quad \sup_{\eta\geq0}\Big\|\frac{d}{d\eta}\big(R_{0}^{+}-R_{0}^{-}\big)(\eta^4+\eta^2)\Big\|_{B(s, -s)}\lesssim 1.$$
	Note that $\Big(M^{\pm}(\eta^4+\eta^2)\Big)^{-1}$ exists in $L^{2}$. Hence the lemma holds by Theorem \ref{Mexpansions} and \ref{continuity}.
\end{proof}

Next, we need only to prove Proposition \ref{middledecay}.  
\begin{lemma}\label{end}
	If 0 is a regular point of $H$,  for $t>1$, then
	\begin{equation}\label{regular1}
	\sup_{x, y\in\mathbf{R}^3}\Bigg|\int_{t^{-1/2}}^{t}e^{-it(\eta^4+\eta^2)}\frac{d}{d\eta}\Big[\big(R_{0}^{+}-R_{0}^{-}\big)VR_{V}^{\pm}VR_{0}^{+}\Big](\eta; x, y)~d\eta\Bigg|\lesssim |t|^{-1/2}
	\end{equation}
and
\begin{equation}\label{regular2}
	\sup_{x, y\in\mathbf{R}^3}\Bigg|\int_{t^{-1/2}}^{t}e^{-it(\eta^4+\eta^2)}\frac{d}{d\eta}\Big[R_{0}^{\pm}V\big(R_{V}^{+}-R_{V}^{-}\big)VR_{0}^{\pm}\Big](\eta; x, y)~d\eta\Bigg|\lesssim |t|^{-1/2}.
	\end{equation}
\end{lemma}
\begin{proof}
	Recall that
	\begin{equation*}	R_{V}^{\pm}=R_{0}^{\pm}-R_{0}^{\pm}v\Big(M^{\pm}(\eta^4+\eta^2)\Big)^{-1}vR_{0}^{\pm}.
	\end{equation*}
For estimate \eqref{regular1}, it is enough to show that
	\begin{equation}\label{i}
	\sup_{x, y\in\mathbf{R}^3}\Bigg|\int_{t^{-1/2}}^{t}e^{-it(\eta^4+\eta^2)}\frac{d}{d\eta}\Big[\big(R_{0}^{+}-R_{0}^{-}\big)VR_{0}^{\pm}VR_{0}^{+}\Big](\eta; x, y)~d\eta\Bigg|\lesssim |t|^{-1/2}
	\end{equation}
	and
	\begin{equation}\label{ii}
	\sup_{x, y\in\mathbf{R}^3}\Bigg|\int_{t^{-1/2}}^{t}e^{-it(\eta^4+\eta^2)}\frac{d}{d\eta}\Big[\big(R_{0}^{+}-R_{0}^{-}\big)VR_{0}^{\pm}v\Big(M^{\pm}\Big)^{-1}vR_{0}^{\pm}VR_{0}^{+}\Big](\eta; x, y)~d\eta\Bigg|\lesssim |t|^{-1/2}.
	\end{equation}
For estimate \eqref{regular2}, it is enough to show that
	\begin{equation}\label{iii}
	\sup_{x, y\in\mathbf{R}^3}\Bigg|\int_{t^{-1/2}}^{t}e^{-it(\eta^4+\eta^2)}\frac{d}{d\eta}\Big[R_{0}^{\pm}V\big(R_{0}^{+}-R_{0}^{-}\big)VR_{0}^{\pm}\Big](\eta; x, y)~d\eta\Bigg|\lesssim |t|^{-1/2}
	\end{equation}
	and
	\begin{equation}\label{iv}
	\sup_{x, y\in\mathbf{R}^3}\Bigg|\int_{t^{-1/2}}^{t}e^{-it(\eta^4+\eta^2)}\frac{d}{d\eta}\Big[R_{0}^{\pm}VR_{0}^{\pm}v\Big((M^{+})^{-1}-(M^{-})^{-1}\Big)vR_{0}^{\pm}VR_{0}^{\pm}\Big](\eta; x, y)~d\eta\Bigg|\lesssim |t|^{-1/2}.
	\end{equation}
Note that estimate \eqref{iii} holds by the similar argument as for estimate \eqref{i}. The estimate \eqref{iv} holds by the similar argument as for estimate \eqref{ii}.	

Applying  the similar arguments to $\frac{d}{d\eta}\Big[\big(R_{0}^{+}-R_{0}^{-}\big)VR_{0}^{\pm}VR_{0}^{+}\Big](\eta; x, y)$ as what we treat the term $\big(R_{0}^{+}-R_{0}^{-}\big)VR_{0}^{\pm}$  in the proof of Lemma \ref{R0VR0}, we can get estimate \eqref{i} holds. Next, we apply the van der Corput's lemma to show \eqref{ii} holds. Denotes
	\begin{align*}
	\mathcal{R}_{M}(\eta; x, y)&=\Big[\big(R_{0}^{+}-R_{0}^{-}\big)VR_{0}^{\pm}v\Big(\frac{d}{d\eta}\big(M^{\pm}\big)^{-1}\Big)vR_{0}^{\pm}VR_{0}^{+}\Big](\eta; x, y);\\
	\mathcal{R}_{1}(\eta; x, y)&=\Big[\Big(\frac{d}{d\eta}\big(R_{0}^{+}-R_{0}^{-}\big)\Big)VR_{0}^{\pm}v\big(M^{\pm}\big)^{-1}vR_{0}^{\pm}VR_{0}^{+}\Big](\eta; x, y);\\
	\mathcal{R}_{2}(\eta; x, y)&=\Big[\big(R_{0}^{+}-R_{0}^{-}\big)V\Big(\frac{d}{d\eta}R_{0}^{\pm}\Big)
	v\big(M^{\pm}\big)^{-1}vR_{0}^{\pm}VR_{0}^{+}\Big](\eta; x, y);\\
	\mathcal{R}_{3}(\eta; x, y)&=\Big[\big(R_{0}^{+}-R_{0}^{-}\big)VR_{0}^{\pm}v\big(M^{\pm}\big)^{-1}v\Big(\frac{d}{d\eta}R_{0}^{\pm}\Big)VR_{0}^{+}\Big](\eta; x, y);\\
	\mathcal{R}_{4}(\eta; x, y)&=\Big[\big(R_{0}^{+}-R_{0}^{-}\big)VR_{0}^{\pm}v\big(M^{\pm}\big)^{-1}vR_{0}^{\pm}V\Big(\frac{d}{d\eta}R_{0}^{+}\Big)\Big](\eta; x, y).
	\end{align*}
	By Lemma \ref{freeresolventuniformlyestimates} and  \ref{middleRV},  for $s>3/2$, then
	\begin{equation*}
	\begin{split}
	\sup_{x, y\in\mathbf{R}^3}\Big|\mathcal{R}_{M}(\eta; x, y)\Big|\lesssim&~ \big\|R_{0}^{+}-R_{0}^{-}\big\|_{L^{2}_{s}\rightarrow L^{\infty}}\big\|VR_{0}^{\pm}v\big\|_{L^{2}\rightarrow L^{2}_{s}}\Big\|\frac{d}{d\eta}\big(M^{\pm}\big)^{-1}\Big\|_{L^{2}\rightarrow L^{2}}\big\|vR_{0}^{\pm}V\big\|_{L^{2}_{-s}\rightarrow L^{2}}\big\|R_{0}^{+}\big\|_{L^{1}\rightarrow L^{2}_{-s}}\\
	\lesssim&~ \big\|R_{0}^{+}-R_{0}^{-}\big\|_{L^{1}\rightarrow L^{\infty}}\big\|(1+|\cdot|)^{s}VR_{0}^{\pm}v\big\|_{L^{2}\rightarrow L^{2}}\Big\|\frac{d}{d\eta}\big(M^{\pm}\big)^{-1}\Big\|_{L^{2}\rightarrow L^{2}}\big\|vR_{0}^{\pm}V(1+|\cdot|)^{s}\big\|_{L^{2}\rightarrow L^{2}}\big\|R_{0}^{+}\big\|_{L^{1}\rightarrow L^{\infty}}\\
	\lesssim &~\frac{\eta(1+\eta)}{(1+2\eta^2)^2}.
	\end{split}
	\end{equation*}
	Similarly, we have
	\begin{align*}
		&~\sup_{x, y\in\mathbf{R}^3}\Big|\mathcal{R}_{j}(\eta; x, y)\Big|\lesssim \frac{\eta(1+\eta)}{(1+2\eta^2)^2},\quad j=2,~ 3; \\
		&~\sup_{x, y\in\mathbf{R}^3}\Big|\mathcal{R}_{j}(\eta; x, y)\Big|\lesssim \frac{1+\eta}{(1+2\eta^2)^2},\quad j=1,~ 4.
	\end{align*}
	Hence $\sup_{\eta>0}\sup_{x, y\in\mathbf{R}^3}\Big|\mathcal{R}_{j}(\eta; x, y)\Big|\lesssim1$ for $j=M, 1, 2, 3, 4.$
   Furthermore, we have
   	\begin{equation*}
   	\begin{split}
   	\sup_{x, y\in\mathbf{R}^3}\Big|\frac{d}{d\eta}\mathcal{R}_{M}(\eta; x, y)\Big|\lesssim&~ \sum_{j=1}^{4}\sup_{x, y\in\mathbf{R}^3}\Big|\mathcal{R}_{j}(\eta; x, y)\Big|+\big\|R_{0}^{+}-R_{0}^{-}\big\|_{L^{1}\rightarrow L^{\infty}}\big\|(1+|\cdot|)^{s}VR_{0}^{\pm}v\big\|_{L^{2}\rightarrow L^{2}}\\
   	&~\times \Big\|\frac{d^2}{d\eta^2}\big(M^{\pm}\big)^{-1}\Big\|_{L^{2}\rightarrow L^{2}}\big\|vR_{0}^{\pm}V(1+|\cdot|)^{s}\big\|_{L^{2}\rightarrow L^{2}}\big\|R_{0}^{+}\big\|_{L^{1}\rightarrow L^{\infty}}\\
   	\lesssim &~\frac{\eta(1+\eta)}{(1+2\eta^2)^2}+\frac{1+\eta}{(1+2\eta^2)^2}\in L^{1}(\mathbf{R}).
   	\end{split}
   	\end{equation*}
	Here, since 0 is a regular point of $H$,  by Theorem \ref{Mexpansions}, then $\big\|\big(M^{\pm}\big)^{-1}\big\|_{L^{2}\rightarrow L^{2}}\lesssim 1$. Thus
	\begin{equation*}
	\begin{split}
	\Big\|\frac{d^2}{d\eta^2}\big(M^{\pm}\big)^{-1}\Big\|_{L^{2}\rightarrow L^{2}}\lesssim&~\Big\|\big(M^{\pm}\big)^{-1}v\Big(\frac{d^2}{d\eta^2}R_{0}^{\pm}\Big)v\big(M^{\pm}\big)^{-1}\Big\|_{L^{2}\rightarrow L^{2}}\\
	&~+	\Big\|\big(M^{\pm}\big)^{-1}v\Big(\frac{d}{d\eta}R_{0}^{\pm}\Big)v\big(M^{\pm}\big)^{-1}v\Big(\frac{d}{d\eta}R_{0}^{\pm}\Big)v\big(M^{\pm}\big)^{-1}\Big\|_{L^{2}\rightarrow L^{2}}\\
	\lesssim&~1+\Big\|\big(M^{\pm}\big)^{-1}\Big\|_{L^{2}\rightarrow L^{2}}\big\|v(x)|x-y|v(y)\big\|_{L^{2}(\mathbf{R}^3)\rightarrow L^{2}(\mathbf{R}^3)}\Big\|\big(M^{\pm}\big)^{-1}\Big\|_{L^{2}\rightarrow L^{2}}	\lesssim 1.
	\end{split}
	\end{equation*}
	For $\mathcal{R}_{j}(\eta; x, y),\,\, j=2, 3$, we also have
	\begin{equation*}
	\sup_{x, y\in\mathbf{R}^3}\Big|\frac{d}{d\eta}\mathcal{R}_{j}(\eta; x, y)\Big|\lesssim\frac{\eta(1+\eta)}{(1+2\eta^2)^2}+\frac{1+\eta}{(1+2\eta^2)^2}\in L^{1}(\mathbf{R}).
	\end{equation*}
	By the van der Corput's lemma, we obtain for $j=M, 2, 3$:
	\begin{equation*}
	\sup_{x, y\in\mathbf{R}^3}\Bigg|\int_{t^{-1/2}}^{t}e^{-it(\eta^4+\eta^2)}\mathcal{R}_{j}(\eta; x, y)~d\eta\Bigg|\lesssim |t|^{-1/2}.
	\end{equation*}
	
	For $\mathcal{R}_{1}(\eta; x, y)$, since
	\begin{equation*}
	\begin{split}
	\mathcal{R}_{1}(\eta; x, y)=&~\int_{\mathbf{R}^{12}}\frac{d}{d\eta}\big(R_{0}^{+}-R_{0}^{-}\big)(\eta; x, x_1)V(x_1)R_{0}^{\pm}(\eta; x_1, x_{2})v(x_2)\big(M^{\pm}\big)^{-1}(\eta; x_{2}, x_{3})v(x_3)\\
	&~\times R_{0}^{\pm}(\eta; x_3, x_4)V(x_4)R_{0}^{+}(\eta; x_4, y)dx_{1}dx_{2}dx_{3}dx_{4}\\
	=&~\int_{\mathbf{R}^{12}}\big(e^{i\eta|x-x_{1}|}+e^{-i\eta|x-x_1|}\big)V(x_1)R_{0}^{\pm}(\eta; x_1, x_{2})v(x_2)\big(M^{\pm}\big)^{-1}(\eta; x_{2}, x_{3})v(x_3)\\
	&~\times R_{0}^{\pm}(\eta; x_3, x_4)V(x_4)R_{0}^{+}(\eta; x_4, y)dx_{1}dx_{2}dx_{3}dx_{4}\\
	:=&~\int_{\mathbf{R}^{12}}\big(e^{i\eta|x-x_{1}|}+e^{-i\eta|x-x_1|}\big)\mathcal{M}(\eta; x_1, x_2, x_3, x_4, y)dx_{1}dx_{2}dx_{3}dx_{4}.
	\end{split}
	\end{equation*}
	Thus, we have
	\begin{equation*}
	\begin{split}
	\Bigg|\int_{t^{-1/2}}^{t}e^{-it(\eta^4+\eta^2)}\mathcal{R}_{1}(\eta; x, y)~d\eta\Bigg|=&~\Bigg|\int_{\mathbf{R}^{12}}\int_{t^{-1/2}}^{t}\Big(e^{-it(\eta^4+\eta^2-\eta|x-x_1|/t)}+e^{-it(\eta^4+\eta^2+\eta|x-x_1|/t)}\Big)\\
	&~\times \mathcal{M}(\eta; x_1, x_2, x_3, x_4, y)~d\eta dx_{1}dx_{2}dx_{3}dx_{4}\Bigg|
	\end{split}
	\end{equation*}
	Then apply the van der Corput's lemma to the two integrals on the right side, we obtain
	\begin{equation*}
	\sup_{x, y\in\mathbf{R}^3}\Bigg|\int_{t^{-1/2}}^{t}e^{-it(\eta^4+\eta^2)}\mathcal{R}_{1}(\eta; x, y)~d\eta\Bigg|\lesssim |t|^{-1/2}.
	\end{equation*}
	For $\mathcal{R}_{4}(\eta; x, y)$, the above estimate also holds by the similar processes as for $\mathcal{R}_{1}(\eta; x, y)$.
\end{proof}

\begin{proof}[\textbf{Proof of Proposition \ref{middledecay}}]
 It is enough to show  that  $\Big[\big(R_{0}^{+}-R_{0}^{-}\big)VR_{V}^{\pm}VR_{0}^{+}\Big](\eta; x, y)$ and $\Big[R_{0}^{\pm}V\big(R_{V}^{+}-R_{V}^{-}\big)VR_{0}^{\pm}\Big](\eta; x, y)$ satisfy the corresponding estimates, since $\Big[R_{0}^{-}VR_{V}^{\pm}V\big(R_{0}^{+}-R_{0}^{-}\big)\Big](\eta; x, y)$ holds by the same arguments as for $\Big[\big(R_{0}^{+}-R_{0}^{-}\big)VR_{V}^{\pm}VR_{0}^{+}\Big](\eta; x, y)$.

 For the term $\big(R_{0}^{+}-R_{0}^{-}\big)VR_{V}^{\pm}VR_{0}^{+}$, integrating by parts, then
 \begin{equation*}
 \begin{split}
 	&\Bigg|\int_{t^{-1/2}}^{t}e^{-it(\eta^4+\eta^2)}\Big[\big(R_{0}^{+}-R_{0}^{-}\big)VR_{V}^{\pm}VR_{0}^{+}\Big](\eta; x, y)~(4\eta^3+2\eta)~d\eta\Bigg|\\
 \lesssim&~|t|^{-2}+\frac{1}{|t|}\Bigg|\Big[\big(R_{0}^{+}-R_{0}^{-}\big)VR_{V}^{\pm}VR_{0}^{+}\Big](t^{-1/2}; x, y)\Bigg|
 	+ \frac{1}{|t|}\Bigg|\int_{t^{-1/2}}^{t}e^{-it(\eta^4+\eta^2)}\frac{d}{d\eta}\Big[\big(R_{0}^{+}-R_{0}^{-}\big)VR_{V}^{\pm}VR_{0}^{+}\Big](\eta; x, y)\Bigg|~d\eta.
 \end{split}
 \end{equation*}
 By Lemma \ref{freeresolventuniformlyestimates}, then
 \begin{equation*}
 \begin{split}
 \Big\|\Big[\big(R_{0}^{+}-R_{0}^{-}\big)VR_{V}^{\pm}VR_{0}^{+}\Big](\eta)\Big\|_{L^1\rightarrow L^{\infty}}
 \leq\,\,\, & \big\|R_{0}^{+}-R_{0}^{-}\big\|_{L^{2}_{s}\rightarrow L^{\infty}}\big\|VR_{V}^{\pm}V\big\|_{L^{2}_{-s}\rightarrow L^{2}_{s}}\big\|R_{0}^{\pm}\big\|_{L^{1}\rightarrow L^{2}_{-s}}\\
 \leq\,\,\, & \big\|R_{0}^{+}-R_{0}^{-}\big\|_{L^1\rightarrow L^{\infty}}\|V\|_{L^{2}_{-s}\rightarrow L^{2}_{s}}\big\|R_{V}^{\pm}\big\|_{L^{2}_{s}\rightarrow L^{2}_{-s}}\|V\|_{L^{2}_{-s}\rightarrow L^{2}_{s}}\big\|R_{0}^{\pm}\big\|_{L^{1}\rightarrow L^{\infty}}\\
 \lesssim\,\,\, & \frac{\eta\big(\eta+\sqrt{1+\eta^2}\big)}{(1+2\eta^2)^2}\big\|R_{V}^{\pm}\big\|_{L^{2}_{s}\rightarrow L^{2}_{-s}}.
 \end{split}
 \end{equation*}
Similarly, we have
	\begin{equation*}
\begin{split}
&\Big\|\frac{d}{d\eta}\Big[\big(R_{0}^{+}-R_{0}^{-}\big)VR_{V}^{\pm}VR_{0}^{+}\Big](\eta)\Big\|_{L^1\rightarrow L^{\infty}}\\
\lesssim \,\,\, &\Big\|\frac{d}{d\eta}\big(R_{0}^{+}-R_{0}^{-}\big)\Big\|_{L^2_{s}\rightarrow L^{\infty}}\|V\|_{L^{2}_{-s}\rightarrow L^{2}_{s}}\big\|R_{V}^{\pm}\big\|_{L^{2}_{s}\rightarrow L^{2}_{-s}}\|V\|_{L^{2}_{-s}\rightarrow L^{2}_{s}}\big\|R_{0}^{\pm}\big\|_{L^{1}\rightarrow L^2_{-s}}\\
&+\big\|R_{0}^{+}-R_{0}^{-}\big\|_{L^2_{s}\rightarrow L^{\infty}}\|V\|_{L^{2}_{-s}\rightarrow L^{2}_{s}}\Big\|\frac{d}{d\eta}R_{V}^{\pm}\Big\|_{L^{2}_{s}\rightarrow L^{2}_{-s}}\|V\|_{L^{2}_{-s}\rightarrow L^{2}_{s}}\big\|R_{0}^{\pm}\big\|_{L^{1}\rightarrow L^{2}_{-s}}\\
&+\big\|R_{0}^{+}-R_{0}^{-}\big\|_{L^2_{s}\rightarrow L^{\infty}}\|V\|_{L^{2}_{-s}\rightarrow L^{2}_{s}}\Big\|R_{V}^{\pm}\Big\|_{L^{2}_{s}\rightarrow L^{2}_{-s}}\|V\|_{L^{2}_{-s}\rightarrow L^{2}_{s}}\Big\|\frac{d}{d\eta}R_{0}^{\pm}\Big\|_{L^{1}\rightarrow L^{2}_{-s}}\\
\lesssim \,\,\, & \frac{\eta+\sqrt{1+\eta^2}}{(1+2\eta^2)^2}\big\|R_{V}^{\pm}\big\|_{L^{2}_{s}\rightarrow L^{2}_{-s}}+\frac{\eta\big(\eta+\sqrt{1+\eta^2}\big)}{(1+2\eta^2)^2}\Big\|\frac{d}{d\eta}R_{V}^{\pm}\Big\|_{L^{2}_{s}\rightarrow L^{2}_{-s}}+\frac{\eta}{(1+2\eta^2)^2}\big\|R_{V}^{\pm}\big\|_{L^{2}_{s}\rightarrow L^{2}_{-s}}.
\end{split}
\end{equation*}
 Thus by Lemma \ref{middleRV}, for $t>1$ we have
 \begin{equation*}
 \frac{1}{|t|}\Bigg|\Big[\big(R_{0}^{+}-R_{0}^{-}\big)VR_{V}^{\pm}VR_{0}^{+}\Big](t^{-1/2}; x, y)\Bigg|\lesssim \begin{cases}
 |t|^{-3/2},\,\, & \,\, 0~ \text{is regular};\\ |t|^{-1},\,\, & \,\, 0~ \text{is a resonance};\\  |t|^{-1/2},\,\, & \,\, 0~ \text{is a resonance and/or eigenvalue}.
 \end{cases}
 \end{equation*}
When 0 is a resonance, then
\begin{equation*}
\frac{1}{|t|}\int_{t^{-1/2}}^{t}\Bigg|\frac{d}{d\eta}\Big[\big(R_{0}^{+}-R_{0}^{-}\big)VR_{V}^{\pm}VR_{0}^{+}\Big](\eta; x, y)\Bigg|~d\eta\lesssim \frac{1}{|t|}\int_{t^{-1/2}}^{t} \frac{\eta+\sqrt{1+\eta^2}}{(1+2\eta^2)^2}\eta^{-1}~d\eta\lesssim |t|^{-1}\ln|t|,\,\,\, t>1.
\end{equation*}
When 0 is a resonance and~/~or an eigenvalue, then
\begin{equation*}
\frac{1}{|t|}\int_{t^{-1/2}}^{t}\Bigg|\frac{d}{d\eta}\Big[\big(R_{0}^{+}-R_{0}^{-}\big)VR_{V}^{\pm}VR_{0}^{+}\Big](\eta; x, y)\Bigg|~d\eta\lesssim \frac{1}{|t|}\int_{t^{-1/2}}^{t} \eta^{-2}~d\eta\lesssim |t|^{-1/2},\,\,\, t>1.
\end{equation*}
However, when 0 is a regular point, we obtain the desired estimate by Lemma \ref{end}.

For the term $R_{0}^{\pm}V\big(R_{V}^{+}-R_{V}^{-}\big)VR_{0}^{\pm}$, integrating by parts, then
 \begin{equation*}
 \begin{split}
 	&\Bigg|\int_{t^{-1/2}}^{t}e^{-it(\eta^4+\eta^2)}\Big[R_{0}^{\pm}V\big(R_{V}^{+}-R_{V}^{-}\big)VR_{0}^{\pm}\Big](\eta; x, y)~(4\eta^3+2\eta)~d\eta\Bigg|\\
 \lesssim&~|t|^{-2}+\frac{1}{|t|}\Bigg|\Big[R_{0}^{\pm}V\big(R_{V}^{+}-R_{V}^{-}\big)VR_{0}^{\pm}\Big](t^{-1/2}; x, y)\Bigg|
 	+ \frac{1}{|t|}\Bigg|\int_{t^{-1/2}}^{t}e^{-it(\eta^4+\eta^2)}\frac{d}{d\eta}\Big[R_{0}^{\pm}V\big(R_{V}^{+}-R_{V}^{-}\big)VR_{0}^{\pm}\Big](\eta; x, y)\Bigg|~d\eta.
 \end{split}
 \end{equation*}
 By Lemma \ref{middleRV}, for $t>1$ we have
 \begin{equation*}
 \frac{1}{|t|}\Bigg|\Big[R_{0}^{\pm}V\big(R_{V}^{+}-R_{V}^{-}\big)VR_{0}^{\pm}\Big](t^{-1/2}; x, y)\Bigg|\lesssim \begin{cases}
 |t|^{-3/2},\,\, & \,\, 0~ \text{is regular};\\ |t|^{-1/2},\,\, & \,\, 0~ \text{is a resonance};\\  |t|^{-1/2},\,\, & \,\, 0~ \text{is a resonance and/or eigenvalue}.
 \end{cases}
\end{equation*}
For the last integral term, when 0 is a regular point, we obtain the desired estimate by Lemma \ref{end}. By Lemma \ref{middleRV}, when 0 is a resonance, then
\begin{equation*}
\frac{1}{|t|}\int_{t^{-1/2}}^{t}\Bigg|\frac{d}{d\eta}\Big[R_{0}^{\pm}V\big(R_{V}^{+}-R_{V}^{-}\big)VR_{0}^{\pm}\Big](\eta; x, y)\Bigg|~d\eta\lesssim \frac{1}{|t|}\int_{t^{-1/2}}^{\infty} \eta^{-2}~d\eta\lesssim |t|^{-1/2},\,\,\, t>1.
\end{equation*}
When 0 is a resonance and~/~or an eigenvalue, then
\begin{equation*}
\frac{1}{|t|}\int_{t^{-1/2}}^{t}\Bigg|\frac{d}{d\eta}\Big[R_{0}^{\pm}V\big(R_{V}^{+}-R_{V}^{-}\big)VR_{0}^{+}\Big](\eta; x, y)\Bigg|~d\eta\lesssim \frac{1}{|t|}\int_{t^{-1/2}}^{\infty} \eta^{-2}~d\eta\lesssim |t|^{-1/2},\,\,\, t>1.
\end{equation*}
\end{proof}

\noindent{\bf Acknowledgments:} The author thanks Prof. Avy Soffer and Prof. Xiaohua Yao for their patient enlightenment. The author is supported by the China  Postdoctoral Science Fundation (No. 2019M653135) and Doctoral Fundation of Chongqing Normal University (No.20XLB018).


\end{document}